\documentclass[11pt]{article}

\usepackage[margin=1in]{geometry}
\usepackage{amsmath,amssymb,mathtools,amsthm}
\usepackage{enumitem}
\usepackage{hyperref}
\usepackage{xcolor}
\usepackage{booktabs}
\usepackage{graphicx}
\usepackage{subcaption}
\usepackage{setspace}

\newcommand{\conv}{\operatorname{conv}}

\newcommand{\R}{\mathbb{R}}
\newcommand{\B}{\{0,1\}}
\newcommand{\eps}{\varepsilon}

\newtheorem{theorem}{Theorem}[section]
\newtheorem{lemma}[theorem]{Lemma}
\newtheorem{proposition}[theorem]{Proposition}
\newtheorem{corollary}[theorem]{Corollary}

\theoremstyle{definition}

\theoremstyle{remark}
\newtheorem{remark}[theorem]{Remark}

\title{Decomposition of Sparse Integer Programs via Nonlinear Edge Encodings and Column-and-Row Generation}
\author{Gustavo Angulo\thanks{gangulo@uc.cl, Department of Industrial and Systems Engineering, Pontificia Universidad Católica de Chile, Santiago, Chile.} \and Santanu S. Dey\thanks{sdey30@gatech.edu, H. Milton Stewart School of Industrial and Systems Engineering, Georgia Institute of Technology, Atlanta, GA, USA.}}
\date{}

\begin{document}

\maketitle

\begin{abstract}
A wide range of sparse integer programs admit a block structure in which subproblems interact through a small set of shared variables. Dualizing the linking equalities yields a decomposable Lagrangian relaxation, but generally introduces a duality gap. Recent work shows that this gap can be closed while preserving decomposability by dualizing exponentially large families of redundant nonlinear consistency constraints on the shared variables. We develop a computational framework for exploiting this idea without explicitly constructing the resulting exponentially large relaxation. Our framework combines nonlinear edge encodings of shared-variable consistency with a column-and-row generation (CRG) algorithm that generates local integer solutions by pricing and encoding constraints by separation. With complete encodings and exact separation, the framework recovers the exact relaxation while retaining independent optimization over the blocks. We introduce several encoding families and establish exponential separations among them: the Generalized family can require exponentially fewer constraints than the Vertex, Monomial, or Reflected families, yet can itself require exponentially many constraints on instances for which a single problem-specific encoding suffices. Computational experiments on decomposed stable-set and dominating-set instances show that CRG substantially outperforms a monolithic formulation across a range of tree topologies and coupling strengths. The results also show that richer encoding families need not perform better computationally, highlighting the choice of encoding as a central issue in effective decomposition.



\end{abstract}

\section{Introduction}
\label{sec:introduction}
A wide range of binary optimization problems arising in operations and engineering exhibit a block-sparse structure: each constraint involves only a small subset of the decision variables, and the variables can be grouped into blocks that interact through a limited number of shared variables; see the discussion in~\cite{bergner2015automatic,dey2018analysis}. This structure can be represented by the intersection graph of the constraint matrix~\cite{dey2018analysis, FaenzaMunozPokutta2022}. A tree decomposition of this graph~\cite{robertson-seymour} reveals a collection of subproblems organized on a tree, where adjacent subproblems interact only through shared variables~\cite{cdx-ipco}.

Introducing a local copy of every variable in each block and enforcing consistency of the copies on the shared variables yields a block-structured formulation organized on a tree $T=(V,E)$,
\begin{equation}\label{eq:block-opt} \mathrm{OPT} := \min\left\{\, \sum_{i\in V} f_i(\mathbf{x}^i) \;:\; \mathbf{x}^i\in X^i\ \ (i\in V),\ \ x_p^{e,i}=x_p^{e,j}\ \ (e=ij\in E,\ p\in Q^e) \,\right\}, 
\end{equation}
where $X^i$ is a local feasible set associated with node $i \in V$, $Q^e$ is the set of variables shared by the two blocks $i, j\in V$ incident to edge $e=ij \in E$, and the equalities enforce consistency of their local copies across the edges of the tree. More details are given in Section~\ref{sec:primal-convexification}.

A leading source of such structure is two-stage and multistage stochastic binary optimization~\cite{birge-louveaux}. In the deterministic equivalent, each scenario defines a local block coupled to a  separate first-stage block through the shared first-stage variables, so the consistency constraints $x_p^{e,i}=x_p^{e,j}$ in~\eqref{eq:block-opt} impose nonanticipativity~\cite{rockafellar-wets}. The tree decomposition $T=(V,E)$ is then a star, with the first-stage block at the center and the scenario blocks as leaves. Multistage problems produce deeper trees.

\paragraph{Decomposability against bound quality.}
Dualizing the coupling equalities $x_p^{e,i}=x_p^{e,j}$ in~\eqref{eq:block-opt} decouples the formulation into one independent subproblem per block, an idea known as dual decomposition, variable splitting, or Lagrangian decomposition~\cite{caroe-schultz,guignard-kim,boland-ph}. The resulting Lagrangian subproblems are solvable in parallel, and valid dual bounds are available at every iteration. The price of this decomposability is a duality gap: because the local sets are nonconvex, the Lagrangian dual generally falls short of $\mathrm{OPT}$, and the resulting bound can be weak. This tension between decomposability and bound quality is one central obstacle in decomposition methods for integer programming. Methods that close the gap typically do so by giving up separability, by restricting the class of problems to which they apply, or by abandoning the availability of intermediate dual bounds.

This tension between decomposability and the duality gap can, however, be resolved by exploiting the fact that constraints that are redundant in the primal formulation need not be redundant in the Lagrangian dual. The paper~\cite{cdx-ipco} pioneered this approach, showing how carefully chosen redundant constraints can be used to strengthen the Lagrangian dual while preserving decomposability. For each edge $e=ij$ of the tree decomposition, the coupling constraints enforce $x_p^{e,i}=x_p^{e,j}$ for $p\in Q^e$. Since the local copies agree at every feasible solution, any function of the shared variables must also agree. Accordingly, for a chosen family of (nonlinear) functions $\mathcal H^e$, the formulation is augmented with redundant constraints of the form $\phi\!\left(\mathbf{x}^{e,i}\right) = \phi\!\left(\mathbf{x}^{e,j}\right)$ for $e=ij\in E$ and $\phi\in \mathcal H^e$. These constraints, henceforth referred to as encoding constraints, are dualized together with the original coupling constraints. The resulting relaxation remains separable across blocks while producing stronger dual bounds. The paper~\cite{cdx-ipco} shows that linear functions do not strengthen the dual and studies two nonlinear encoding families: monomial constraints, which enforce consistency of products of shared variables and induce the M-Lagrangian dual, and vertex-indicator constraints, which enforce consistency of all assignments of the shared variables and induce the V-Lagrangian dual. Both preserve decomposability and recover strong duality when enforced exhaustively on a tree decomposition. The resulting bound admits a Geoffrion-style primal characterization~\cite{geoffrion} as the value of a convex program obtained by replacing each local set with the convex hull of its lifted representation while leaving the coupling constraints explicit.

\paragraph{A zero-gap Lagrangian dual is not yet an algorithm.}
This resolves the tension in principle but not in computation, for two reasons. First, exact encoding requires exponentially many consistency constraints in the number of shared variables, making an explicit formulation of the full strengthened dual impractical. Second, when working with only a subset of the encoding constraints, the nonlinear Lagrangian dual is naturally optimized by subgradient or bundle methods, which return dual multipliers but no primal solution from which additional violated encoding constraints can be separated directly; recovering one is possible but requires additional machinery~\cite{anstreicher2009two}. In practice, one must therefore work with a fixed, strict subfamily chosen before the solve begins; see, for example, the computations in~\cite{cdx-ipco}. In the case of the M-Lagrangian dual, for example, this corresponds to an apriori truncation by degree, and the resulting bound depends on that truncation rather than adapting to the inequalities required by the particular instance.

\paragraph{This paper.}
We treat the primal characterization of the Lagrangian dual as the computational object and solve the corresponding convex program directly via Dantzig--Wolfe decomposition~\cite{dantzig-wolfe}. Columns correspond to feasible lifted integer solutions of the local blocks, while rows correspond to consistency constraints: the original equalities on shared variables together with the nonlinear encoding constraints induced by the chosen edge encodings. Since both the column set and the encoding constraints can be exponential in size, we generate columns by pricing and constraints by separation, yielding a column-and-row generation (CRG) framework.

A key consequence of working in the primal is that the encoding family enters only through a separation oracle, avoiding explicit instantiation of the exponentially large family of encoding constraints. Instead, violated constraints are identified dynamically and added to the master as needed, allowing the algorithm to converge to the relaxation defined by the complete family while materializing only a small fraction of it.

This yields, to the best of our knowledge, the first decomposition algorithm for integer programs that simultaneously has all three of the following properties. It is \emph{exact}: with a complete encoding family and exact separation, the algorithm terminates with the optimal value of~\eqref{eq:block-opt}. It is \emph{decomposable}: all subproblem work factors into one independent pricing problem per block, solvable in parallel, and a valid dual bound on $\mathrm{OPT}$ is available at the end of every outer iteration. And it is \emph{structurally general}: it requires only that the block interaction graph be a tree and that the shared variables be binary. No convexity, continuous recourse, relatively complete recourse, or two-stage structure is assumed of the blocks.


\paragraph{Encoding choice is a modeling decision.}
Because the encoding family enters only through separation, the framework accommodates any family for which violated constraints can be identified. This leads to the question of how the choice of encoding family affects the method, which we investigate both theoretically and computationally. We organize the consistency encoding constraints through a general class of nonlinear edge encodings containing the Vertex and Monomial constructions of~\cite{cdx-ipco} together with two further families, Reflected and Generalized. Within this class we establish exponential separations: the Generalized family strictly dominates the other three, the Monomial and Reflected families are incomparable to one another, and even the Generalized family is not universal, since it can require exponentially many rows on a polynomially tractable instance that a single encoding outside the class closes. The proofs rest on classical polynomial threshold degree lower bounds~\cite{minsky-papert} together with random restriction arguments.

\paragraph{Contributions.}
Therefore, the contributions of this paper are as follows.
\begin{enumerate}
\item \emph{An exact, decomposable, and structurally general decomposition algorithm.} We develop a column-and-row generation (CRG) framework that computes the primal characterization of the encoding-strengthened Lagrangian bounds of~\cite{cdx-ipco}. The method combines Dantzig--Wolfe column generation with edge-based separation routines that dynamically identify and add violated encoding constraints, and it delivers the three properties stated above.

\item \emph{A unified encoding framework.} We organize the generated consistency constraints through nonlinear edge encodings and introduce a Generalized encoding family that subsumes the Vertex and Monomial constructions of~\cite{cdx-ipco} and the new Reflected family.

\item \emph{Theoretical limits of encoding families.} We establish that the Monomial, Reflected, and Vertex families can each require exponentially more rows than the Generalized family to close the gap, so the Generalized family strictly dominates all three. The Monomial and Reflected families are moreover not comparable to each other: one can require exponentially many rows where the other needs only polynomially many, and vice versa. Yet even the Generalized family is not universal: it can require exponentially many rows on a polynomially tractable instance for which a single encoding outside the above families closes the gap. 

\item \emph{Computational evaluation.} We evaluate the proposed encodings on decomposed stable-set and dominating-set instances over star, path, binary-tree, and random-tree topologies, three coupling levels, and both deterministic and stochastic objectives. Beyond the aggregate comparison, we analyze the structure of the generated rows and find that the families producing the most rows are not those producing the best bounds: Reflected and Generalized exhaust their row budget with dense, near-duplicate rows, while Monomial uses a fraction of the budget and gets more out of it.\end{enumerate}

\paragraph{Related decomposition approaches.}
We position our approach within the broader literature on decomposition for integer programming by comparing existing methods along three dimensions: exactness, decomposability, and structural generality.
\begin{itemize}
\item \emph{Inexact methods}: Lagrangian and dual decomposition~\cite{caroe-schultz,guignard-kim} and progressive hedging~\cite{boland-ph} retain decomposability and intermediate dual bounds but leave a duality gap.
\item \emph{Methods requiring a master--recourse structure.} Benders-type methods achieve exactness by constructing an approximation of a recourse value function over the master variables. For two-stage stochastic integer programs with a pure binary first stage, the integer L-shaped method~\cite{benders,laporte-louveaux,angulo-ahmed-dey} is exact and decomposable. A substantial subsequent literature develops stronger cuts for stochastic integer programs, including split cuts derived in an extended space and projected into the Benders reformulation~\cite{bdgl-ijoc}, as well as Lagrangian and strengthened Benders cuts~\cite{zas-sddip,cl-ijoc,acp-lipschitz,deng-xie}; related constructions yield finite convergence for multistage problems with binary state variables. These methods are particularly close in spirit to ours: they also seek to close the gap arising from decomposition, and Lagrangian-cut generation invokes essentially the same local integer-optimization oracle that appears in our pricing step. Scenario decomposition~\cite{ahmed-scenario} achieves exactness by a different mechanism, enumerating first-stage solutions rather than approximating a value function by cuts.

The structural distinction is that these approaches exploit a directed master–recourse organization. For example, in multistage methods such as SDDiP, this takes the form of a sequential stagewise structure, with downstream decisions and value functions conditioned on the state inherited from earlier stages.
In contrast, formulation~\eqref{eq:block-opt} does not designate any block as a master and requires neither a stage ordering nor a recourse value function. Its blocks may be arranged on an arbitrary tree and interact symmetrically through shared variables. In particular, all block subproblems are independent and can be solved fully in parallel.

\item \emph{Non-decomposable approaches.} Augmented Lagrangian~\cite{feizollahi-al,gu-al} and semi-Lagrangian~\cite{beltran-sl} duals attain strong duality at the expense of the decomposable structure, which~\cite{sun-decomp} restores within an augmented Lagrangian framework through a mechanism distinct from the present one.
\item \emph{Other methods.} Nonserial dynamic programming~\cite{bertele-brioschi} and LP formulations for polynomial optimization over tree-structured supports~\cite{bienstock-munoz} achieve exactness on the same sparsity structure, but require an exponentially large object. In addition, nonserial dynamic programming does not generally provide a useful bound if terminated before completion. In contrast, our method does not require an exponential-size formulation upfront and produces a valid dual bound after every completed outer iteration, providing meaningful information even when terminated before convergence.
\end{itemize}
On the mechanics, column-and-row generation builds on Dantzig--Wolfe decomposition~\cite{dantzig-wolfe,luebbecke-cg}: columns are generated by pricing, while the convexified master is progressively strengthened by dynamically generated coupling rows. Unlike branch-and-price~\cite{barnhart-bp}, exactness is obtained without branching, through the generation of sufficiently strong encoding constraints while preserving blockwise pricing. 
The row-generation component plays a role analogous to a hierarchy of block-decomposable relaxations such as DRL*~\cite{minoux-ouzia}, and is related in motivation, though not in mechanism, to the lift-and-project, RLT, and Lasserre hierarchies~\cite{sherali-adams,lasserre,rothvoss}. 
\paragraph{Outline.}
Section~\ref{sec:primal-convexification} sets up the block-structured formulation, the edge encodings, and the convexified primal characterization. The column-and-row generation algorithm is developed in Section~\ref{sec:row-column-generation}. Section~\ref{sec:strategies} describes the four encoding strategies, and Section~\ref{sec:encoding-comparison} formally establishes their strengths. Section~\ref{sec:test-instances} describes the test instances and computational setup, while results are discussed in Section~\ref{sec:computational-results}. Finally, conclusions are drawn in Section~\ref{sec:discussion}.

\section{Primal convexification by edge encodings}
\label{sec:primal-convexification}

We consider a block-structured binary optimization problem whose blocks are organized by a tree $T=(V,E)$. Each node $i\in V$ represents a local feasible set $X^i$ with local variables $\mathbf{x}^i$ and local cost function $f_i$. The local vector $\mathbf{x}^i$ contains the local copy of the variables that belong to block $i$, including the variables shared with neighboring blocks. For each edge $e=ij\in E$, let $Q^e = Q^{ij}$ denote the set of variables shared by blocks $i$ and $j$. The nodes $i$ and $j$ are called the endpoints of $e$ and $Q^e$ is called the boundary of the edge $e$. We denote by $\mathbf{x}^{e,i}$ and $\mathbf{x}^{e,j}$ the boundary subvectors of $\mathbf{x}^i$ and $\mathbf{x}^j$, respectively, indexed by $Q^e$ and written in a common order. Components are denoted by $x_p^{e,i}$ and $x_p^{e,j}$ for $p\in Q^e$. The decomposed formulation is
\begin{equation}
\label{eq:decomposed-problem}
\begin{aligned}
    \mathrm{OPT}:=\min \quad
        & \sum_{i\in V} f_i(\mathbf{x}^i) \\
    \text{s.t.}\quad
        & \mathbf{x}^i\in X^i, && i\in V,\\
        & x_p^{e,i}=x_p^{e,j}, && e=ij\in E,\ p\in Q^e.
\end{aligned}
\end{equation}
Throughout, we assume that all variables appearing in coupling constraints are binary, that is, $X^i$ includes the constraints that $x_p^{e,i} \in \{0,1\}$ for all $e=ij\in E,\ p\in Q^e$. Local variables that are not shared with any neighbor are not constrained to be binary.

The purpose of the strengthened Lagrangian reformulations is to add redundant constraints that are implied by the coupling equations $x_p^{e,i}=x_p^{e,j}$ in \eqref{eq:decomposed-problem}, but improve the Lagrangian dual bound. For each edge $e=ij\in E$, let $\mathcal H^e$ be an index set of \emph{encoding functions} $\phi_h^e : \B^{Q^e}\to \R$, $h\in\mathcal H^e$. The corresponding \emph{encoding vector} associated with block $i$ on edge $e=ij$ is denoted by $\mathbf{w}^{e,i}$, with scalar components $w_h^{e,i}=\phi_h^e(\mathbf{x}^{e,i})$, $h\in\mathcal H^e$. Analogously, block $j$ has a vector $\mathbf{w}^{e,j}$ with entries $w_h^{e,j}=\phi_h^e(\mathbf{x}^{e,j})$. Since  $\mathbf{x}^{e,i}=\mathbf{x}^{e,j}$ in every feasible solution of \eqref{eq:decomposed-problem}, the equalities
$w_h^{e,i}=w_h^{e,j}$, $e=ij\in E$, $h\in\mathcal H^e$,
are redundant for the original integer formulation. The resulting formulation is:
\begin{equation}
\label{eq:decomposed-problem1}
\begin{aligned}
   \min \quad
        & \sum_{i\in V} f_i(\mathbf{x}^i) \\
    \text{s.t.}\quad
        & \mathbf{x}^i\in X^i, && i\in V,\\
        & x_p^{e,i}=x_p^{e,j}, && e=ij\in E,\ p\in Q^e,\\
        &\phi_h^e(\mathbf{x}^{e,i}) = \phi_h^e(\mathbf{x}^{e,j}), && e=ij \in E, h \in \mathcal{H}^e.
\end{aligned}
\end{equation}
The redundant constraint, when dualized, may strengthen the Lagrangian dual as discussed next~\cite{cdx-ipco}.

Given a collection of encoding families $\mathcal H=\{\mathcal H^e\}_{e\in E}$, define the lifted local set
\begin{equation*}
X_{\mathcal H}^i:= \left\{ (\mathbf{x}^i,\mathbf{w}^i) \left| \begin{array}{l}
\mathbf{x}^i\in X^i,\\
\mathbf{w}^i=\{\mathbf{w}^{e,i}: e\in\delta(i)\},\\
w_h^{e,i}=\phi_h^e(\mathbf{x}^{e,i}),
\quad e\in\delta(i),\ h\in\mathcal H^e
\end{array} \right. \right\},
\end{equation*}
where $\delta(i)$ is the set of edges incident to $i$. The primal characterization of the encoding-strengthened Lagrangian dual is the following convex relaxation:
\begin{equation}
\label{eq:primal-characterization}
\begin{aligned}
    \mathrm{OPT}(\mathcal H):=\min \quad
        & \sum_{i\in V} \theta_i\\
    \text{s.t.}\quad
        & (\mathbf{x}^i,\mathbf{w}^i,\theta_i)\in
          \conv\bigl\{(\mathbf{x},\mathbf{w},f_i(\mathbf{x})):
          (\mathbf{x},\mathbf{w})\in X_{\mathcal H}^i\bigr\},
            && i\in V,\\
        & x_p^{e,i}=x_p^{e,j},
            && e=ij\in E,\ p\in Q^e,\\
        & w_h^{e,i}=w_h^{e,j},
            && e=ij\in E,\ h\in\mathcal H^e,
\end{aligned}
\end{equation}
where $w_h^{e,i}=w_h^{e,j}$ are the encoding constraints. It is straightforward to see that Problem \eqref{eq:primal-characterization} is a relaxation of \eqref{eq:decomposed-problem}. Since we use a minimization convention, this implies $\mathrm{OPT}(\mathcal H)$ is a lower bound on $\mathrm{OPT}$. Note that, by projection, the convexification in \eqref{eq:primal-characterization} can be restricted to shared variables only.

When the encoding family is suitably selected and the block structure satisfies the tree assumptions, this relaxation is exact~\cite{cdx-ipco}. In particular, the paper~\cite{cdx-ipco} introduces two specific encoding families, termed the M- and V-encodings, which we formally define in Section~\ref{sec:strategies}. The following theorem summarizes their main exactness result.
\begin{theorem}[{\cite{cdx-ipco}}]\label{thm:mvstrong}
Suppose that the block interaction graph is a tree. If the encoding family $\mathcal H$ is chosen as either the complete M-encoding or the V-encoding, then $\mathrm{OPT}(\mathcal H) =\mathrm{OPT}$.
\end{theorem}

The Dantzig-Wolfe representation of \eqref{eq:primal-characterization} is obtained by writing each point of $\conv(X_{\mathcal H}^i)$ as a convex combination of local integer columns. Let $\mathcal P^i$ be the set of feasible local columns of block $i$. A column $k\in\mathcal P^i$ consists of a local vector $\mathbf{x}^{i,k}\in X^i$, its restrictions (or subvector) $\mathbf{x}^{e,i,k}$ to the variables shared with neighboring blocks, for $e\in\delta(i)$, and the induced encoding entries $w_h^{e,i,k}=\phi_h^e(\mathbf{x}^{e,i,k})$ for $e\in\delta(i)$ and $h\in\mathcal H^e$. Let $c_{ik}:=f_i(\mathbf{x}^{i,k})$ be the scalar cost of column $k$ of block $i$, and let $\lambda_{ik}$ be its convex weight. The complete master problem is
\begin{equation}
\label{eq:complete-master}
\begin{aligned}
    \min \quad
        & \sum_{i\in V}\sum_{k\in\mathcal P^i} c_{ik}\lambda_{ik}\\
    \text{s.t.}\quad
        & \sum_{k\in\mathcal P^i}\lambda_{ik}=1,
            && i\in V,\\
        & \sum_{k\in\mathcal P^i}x_p^{e,i,k}\lambda_{ik}
          -\sum_{k\in\mathcal P^j}x_p^{e,j,k}\lambda_{jk}=0,
            && e=ij\in E,\ p\in Q^e,\\
        & \sum_{k\in\mathcal P^i}w_h^{e,i,k}\lambda_{ik}
          -\sum_{k\in\mathcal P^j}w_h^{e,j,k}\lambda_{jk}=0,
            && e=ij\in E,\ h\in\mathcal H^e,\\
        & \lambda_{ik}\ge 0,
            && i\in V,\ k\in\mathcal P^i.
\end{aligned}
\end{equation}
The first constraints are the convexity constraints. The second group enforces consistency of the original shared variables. The third group enforces consistency of the selected edge encodings, that is, these are the encoding constraints. In exact formulations, either the column sets $\mathcal P^i$, the encoding sets $\mathcal H^e$, or both are exponentially large. The implementation therefore solves \eqref{eq:complete-master} by generating columns and encoding rows dynamically as discussed in the next section.

\section{Column-and-row generation algorithm}
\label{sec:row-column-generation}

At any iteration, the algorithm maintains a restricted set of columns $\mathcal K^i\subseteq\mathcal P^i$ for each block and a restricted set of encoding rows $\mathcal C^e\subseteq\mathcal H^e$ for each edge. The restricted master problem is obtained from \eqref{eq:complete-master} by replacing $\mathcal P^i$ by $\mathcal K^i$ and $\mathcal H^e$ by $\mathcal C^e$:
\begin{equation}
\label{eq:rmp}
\begin{aligned}
    z_{\mathrm{RMP}}:=\min \quad
        & \sum_{i\in V}\sum_{k\in\mathcal K^i} c_{ik}\lambda_{ik}\\
    \text{s.t.}\quad
        & \sum_{k\in\mathcal K^i}\lambda_{ik}=1,
            && i\in V,\\
        & \sum_{k\in\mathcal K^i}x_p^{e,i,k}\lambda_{ik}
          -\sum_{k\in\mathcal K^j}x_p^{e,j,k}\lambda_{jk}=0,
            && e=ij\in E,\ p\in Q^e,\\
        & \sum_{k\in\mathcal K^i}w_h^{e,i,k}\lambda_{ik}
          -\sum_{k\in\mathcal K^j}w_h^{e,j,k}\lambda_{jk}=0,
            && e=ij\in E,\ h\in\mathcal C^e,\\
        & \lambda_{ik}\ge 0,
            && i\in V,\ k\in\mathcal K^i.
\end{aligned}
\end{equation}
The set $\mathcal C^e$ is initially empty or small, so that the first restricted master may contain only the original linear coupling rows. As the algorithm proceeds, new encoding rows are added to $\mathcal C^e$, and future columns receive their coefficients through the same functions $\phi_h^e$.

\subsection{Column generation for fixed encoding rows}
\label{subsec:pricing}

Let $\alpha_i$ be the dual multiplier of the convexity constraint of block $i$. For each edge $e=ij\in E$, let $\pi_p^e$ be the dual multiplier of the original coupling row for component $p\in Q^e$, and let $\mu_h^e$ be the dual multiplier of the generated encoding row indexed by $h\in\mathcal C^e$. We orient each edge as written, so the corresponding rows in \eqref{eq:rmp} have the form ``block $i$ minus block $j$''. Define
$$\tau_i^e:=
    \begin{cases}
        +1, & \text{if } e=ij \text{ and } i \text{ is the first endpoint},\\
        -1, & \text{if } e=ji \text{ and } i \text{ is the second endpoint}.
    \end{cases}
$$
For a fixed row set $\mathcal C=\{\mathcal C^e\}_{e\in E}$, the pricing problem for block $i$ is
\begin{equation}
\label{eq:pricing}
\rho_i:= \min_{\mathbf{x}^i\in X^i} \left\{
\begin{aligned}
&f_i(\mathbf{x}^i)-\alpha_i\\
&\quad- \sum_{e\in\delta(i)}\tau_i^e \left(\sum_{p\in Q^e}\pi_p^e x_p^{e,i} +\sum_{h\in\mathcal  C^e}\mu_h^e\phi_h^e(\mathbf{x}^{e,i}) \right)
\end{aligned}
\right\}.
\end{equation}
The value $\rho_i$ is the minimum reduced cost among all columns of block $i$ with respect to the current master. If $\rho_i< -\eps_{\mathrm{col}}$, an optimizer of \eqref{eq:pricing} defines a new column and is added to $\mathcal K^i$. The column-generation phase for the current row set terminates when $\rho_i\ge -\eps_{\mathrm{col}}$ for every $i\in V$.

Solving the pricing problem once and adding any negative reduced-cost columns to the restricted master constitutes one \emph{inner iteration}. When no negative reduced-cost column remains, \eqref{eq:rmp} solves the full Dantzig-Wolfe master corresponding to the current encoding row set $\mathcal C$, thereby completing one \emph{outer iteration}. Since this master is a relaxation of~\eqref{eq:block-opt}, its optimal objective value is a valid lower bound on~$\mathrm{OPT}$. The next outer iteration begins by separating violated encoding constraints and adding those that are identified, as described in the next section. The updated master is then re-solved via column generation, completing the outer iteration.

\subsection{Row generation by encoding function separation}
\label{subsec:row-generation}

After column generation converges for the current encoding constraint  set, the algorithm checks whether the current master solution violates any encoding equality that has not yet been added. For an edge $e=ij\in E$, the master solution assigns a nonnegative weight $\lambda_{ik}$ to each column of block $i$, with the weights summing to one. Thus, the master solution can be viewed as a probability distribution over the columns of each block. For a signature $\mathbf{s}\in\B^{Q^e}$ of the shared variables $Q^e$, let $\nu_{\mathbf{s}}^{e,i}$ denote the total weight of columns of block $i$ whose shared variables take the value $\mathbf{s}$. Similarly, define $\nu_{\mathbf{s}}^{e,j}$ for block $j$. We refer to the collections $\nu^{e,i}$ and $\nu^{e,j}$ as the boundary distributions associated with edge $e$:
\begin{equation}
\label{eq:marginals}
    \nu_{\mathbf{s}}^{e,i} :=\sum_{\substack{k\in\mathcal K^i:\,\mathbf{x}^{e,i,k}=\mathbf{s}}} \lambda_{ik}, \qquad
    \nu_{\mathbf{s}}^{e,j} :=\sum_{\substack{k\in\mathcal K^j:\,\mathbf{x}^{e,j,k}=\mathbf{s}}} \lambda_{jk}, \qquad \mathbf{s}\in\B^{Q^e}.
\end{equation}

For a candidate encoding $h\in\mathcal H^e$, the violation of its consistency row is the difference $\Delta_h^e := \sum_{\mathbf{s}\in\B^{Q^e}} \phi_h^e(\mathbf{s}) \left(\nu_{\mathbf{s}}^{e,i}-\nu_{\mathbf{s}}^{e,j}\right)$. If $|\Delta_h^e|>\eps_{\mathrm{row}}$ and $h\notin\mathcal C^e$, the algorithm adds the row
\begin{equation}
\label{eq:generated-row}
    \sum_{k\in\mathcal K^i}\phi_h^e(\mathbf{x}^{e,i,k})\lambda_{ik}  -   \sum_{k\in\mathcal K^j}\phi_h^e(\mathbf{x}^{e,j,k})\lambda_{jk}=0.
\end{equation}
The separation oracle need not identify all violated rows. It may instead return only a subset, such as the most violated encoding rows on each edge or a fixed number of encoding rows per iteration.

\subsection{Overall algorithm}
\label{subsec:algorithm-flow}

The resulting procedure alternates between column generation (inner iteration) and row generation (outer iteration). A high-level description is given below.
\begin{enumerate}[label=\arabic*.]
\item Initialize nonempty column sets $\mathcal K^i$ for all $i\in V$ and initial encoding sets $\mathcal C^e$ for all $e\in E$. In the implementation, initial columns may be obtained from a monolithic solution, when available, by projecting it onto the blocks.

\item Solve the restricted master problem \eqref{eq:rmp} and extract the dual multipliers $\alpha_i$, $\pi_p^e$, and $\mu_h^e$.

\item For each block $i\in V$, solve the pricing problem     \eqref{eq:pricing}. Add every column whose reduced cost is smaller than $-\eps_{\mathrm{col}}$.

\item Repeat Steps 2--3 until no block generates a negative reduced-cost column.

\item For every edge $e=ij\in E$, compute the boundary distributions \eqref{eq:marginals}. Apply the selected separation strategy to find encodings $h\in\mathcal H^e\setminus\mathcal C^e$ with    $|\Delta_h^e|>\eps_{\mathrm{row}}$.

\item Add the selected rows \eqref{eq:generated-row} to the master and update the pricing models so that the corresponding terms appear in \eqref{eq:pricing}.

\item Terminate if a full outer iteration adds neither rows nor columns in the corresponding inner iteration, or if an external stopping rule such as a time limit or relative gap tolerance is met. Otherwise, return to Step 2.
\end{enumerate}

When the algorithm terminates because neither columns nor rows corresponding to encoding constraints are generated, then the current master solves exactly the relaxation defined by the generated encoding constraints. If the separation routine is exact for a complete encoding family which satisfies the assumptions for the strong-duality Theorem~\ref{thm:mvstrong}, this relaxation is exact and the resulting bound equals OPT. Otherwise, the algorithm returns the bound corresponding to the generated subset of encoding constraints.

\section{Specific encoding strategies}
\label{sec:strategies}

The master problem and pricing framework are agnostic to the choice of encoding. An encoding strategy is specified by a family of functions $\phi_h^e$, a separation oracle that identifies violated encoding constraints, and a pricing formulation that evaluates $\phi_h^e(\mathbf{x}^{e,i})$ in \eqref{eq:pricing}. We describe four strategies: Vertex, Monomial, Reflected, and Generalized.

\subsection{Vertex strategy}
\label{subsec:v-strategy}

The Vertex (or V) encoding was introduced in~\cite{cdx-ipco}. For each signature vector $\mathbf{s}\in\B^{Q^e}$, define $\phi_{\mathbf{s}}^{V,e}(\mathbf{x}^{e,i}):=\mathbf{1}\{\mathbf{x}^{e,i}=\mathbf{s}\}$, where $\mathbf{1}\{\cdot\}$ is the indicator function. The corresponding encoding row is 
$$\sum_{k\in\mathcal K^i}\mathbf{1}\{\mathbf{x}^{e,i,k}=\mathbf{s}\}\lambda_{ik} - \sum_{k\in\mathcal K^j}\mathbf{1}\{ \mathbf{x}^{e,j,k}=\mathbf{s}\}\lambda_{jk}=0,$$
or, equivalently, $\nu_{\mathbf{s}}^{e,i}=\nu_{\mathbf{s}}^{e,j}$. Thus, separation is immediate: compute the two marginal distributions on edge $e$ and add any signatures for which $|\nu_{\mathbf{s}}^{e,i}-\nu_{\mathbf{s}}^{e,j}|>\eps_{\mathrm{row}}$.

In the pricing problem, the indicator $\mathbf{1}\{\mathbf{x}^{e,i}=\mathbf{s}\}$ can be represented by a binary variable $w_{\mathbf{s}}^{e,i}$. For a fixed signature $\mathbf{s}$, the standard linearization is
\begin{align*}
    w_{\mathbf{s}}^{e,i} &\le x_p^{e,i},
        && p\in Q^e:\ s_p=1, \\
    w_{\mathbf{s}}^{e,i} &\le 1-x_p^{e,i},
        && p\in Q^e:\ s_p=0, \\
    w_{\mathbf{s}}^{e,i}
        &\ge
        \sum_{p\in Q^e:\,s_p=1}x_p^{e,i}
        +\sum_{p\in Q^e:\,s_p=0}(1-x_p^{e,i})
        -|Q^e|+1. \label{eq:v-lin-lower}
\end{align*}

\subsection{Monomial strategy}
\label{subsec:m-strategy}
The Monomial (or M) encoding was introduced in~\cite{cdx-ipco}. The Monomial strategy uses monomial encodings on the shared variables. For each subset $S\subseteq Q^e$, define $\phi_S^{\mathrm{M},e}(\mathbf{x}^{e,i}):=\prod_{p\in S}x_p^{e,i},$ with the convention that $\phi_\emptyset^{\mathrm{M},e}\equiv 1$. The row associated with $S$ enforces equality of the corresponding ``moment":
\begin{equation*}
    \sum_{k\in\mathcal K^i}
        \left(\prod_{p\in S}x_p^{e,i,k}\right)\lambda_{ik}
    -
    \sum_{k\in\mathcal K^j}
        \left(\prod_{p\in S}x_p^{e,j,k}\right)\lambda_{jk}=0.
\end{equation*}
In the exact version of the Monomial strategy, separation is over all subsets $S\subseteq Q^e$. Given the current master solution, the most violated row in one direction can be found by solving
\begin{equation}
\label{eq:m-separation-objective}
    \max_{S\subseteq Q^e}
    \left\{
    \sum_{k\in\mathcal K^i}\lambda_{ik} w_S^{e,i,k}
    -
    \sum_{k\in\mathcal K^j}\lambda_{jk} w_S^{e,j,k}
    \right\},
\end{equation}
where $w_S^{e,i,k}:=\prod_{p\in S}x_p^{e,i,k}$ and $w_S^{e,j,k}:=\prod_{p\in S}x_p^{e,j,k}$. A symmetric problem with the two endpoints interchanged is also solved, that is maximizing $\sum_{k\in\mathcal K^j}\lambda_{jk} w_S^{e,j,k}-\sum_{k\in\mathcal K^i}\lambda_{ik} w_S^{e,i,k}$, and corresponding rows with  violation are added.

A compact integer programming model for \eqref{eq:m-separation-objective} uses a binary variable $z_p^e$ to indicate whether $p$ is selected in $S$, and binary variables $u_k^{e,i}$ and $u_k^{e,j}$ to indicate whether the selected monomial is active in column $k$ on the two endpoints. One such model is:
\begin{align*}
    \max\quad
        & \sum_{k\in\mathcal K^i}\lambda_{ik}u_k^{e,i}
          -\sum_{k\in\mathcal K^j}\lambda_{jk}u_k^{e,j} \\
    \text{s.t.}\quad
        & u_k^{e,l}\le x_p^{e,l,k}+1-z_p^e,
            && l\in\{i,j\},\ k\in\mathcal K^l,\ p\in Q^e, \\
        & \sum_{p\in Q^e}x_p^{e,l,k}z_p^e
          \le \sum_{p\in Q^e}z_p^e-1+u_k^{e,l},
            && l\in\{i,j\},\ k\in\mathcal K^l, \\
        & z_p^e\in\B,
            && p\in Q^e,\\
        & u_k^{e,l}\in\B,
            && l\in\{i,j\},\ k\in\mathcal K^l.
\end{align*}

In the pricing problem, once a subset $S$ has been selected, the monomial is linearized by introducing $w_S^{e,i}$ and imposing
\begin{align*}
    w_S^{e,i} &\le x_p^{e,i},
        && p\in S, \\
    w_S^{e,i} &\ge \sum_{p\in S}x_p^{e,i}-|S|+1.
\end{align*}

\subsection{Reflected strategy}
\label{subsec:reflected-strategy}

The Reflected strategy applies the Monomial strategy to complements of the shared variables. For each subset $S\subseteq Q^e$, define $\phi_S^{\mathrm{R},e}(\mathbf{x}^{e,i}):=\prod_{p\in S}(1-x_p^{e,i})$, with the convention that $\phi_\emptyset^{\mathrm{R},e}\equiv 1$. The associated row matches the mass of columns in which all variables in $S$ are equal to zero:
\begin{equation*}
\label{eq:reflected-row}
    \sum_{k\in\mathcal K^i}
        \left(\prod_{p\in S}(1-x_p^{e,i,k})\right)\lambda_{ik}
    -
    \sum_{k\in\mathcal K^j}
        \left(\prod_{p\in S}(1-x_p^{e,j,k})\right)\lambda_{jk}=0.
\end{equation*}
Separation is identical to the Monomial separation problem after replacing every constant $x_p^{e,i,k}$ by $1-x_p^{e,i,k}$ and every $x_p^{e,j,k}$ by $1-x_p^{e,j,k}$. Thus the same integer programming separation model can be used on complemented boundary signatures. The modification of the pricing problem is analogous.


\paragraph{Connection to Monomial encoding.} Observe the following standard result.
\begin{lemma}\label{lem:refM}
For every subset $S \subseteq Q^e$, we have that $\phi_S^{\mathrm{M},e}(x) = \sum_{U\subseteq S} (-1)^{|U|}\phi_U^{\mathrm{R},e}(x)$ and $\phi_S^{\mathrm{R},e}(x) = \sum_{U\subseteq S}(-1)^{|U|}\phi_U^{\mathrm{M},e}(x)$. Therefore, the Monomial and Reflected encoding families span the same vector space of functions on $\{0,1\}^{Q^e}$.
\end{lemma}
\begin{proof}
Both identities follow by expanding $\prod_{p\in S}(1-(1-x_p))$ and $\prod_{p\in S}(1-x_p)$ and collecting terms by inclusion--exclusion.
\end{proof}

In other words, every monomial can be expressed as a linear combination of reflected monomials, and vice versa. Therefore, any term of the form $\phi_S^e(x)$ appearing in the objective function of the Lagrangian relaxation of~\eqref{eq:decomposed-problem1}, when using monomial encoding, can be represented exactly using reflected monomial terms. Since the monomial and reflected encoding families span the same function space on $\{0,1\}^{Q^e}$, Theorem~\ref{thm:mvstrong} immediately yields the following corollary.

\begin{corollary}
\label{cor:reflected-strong}
Suppose that the block interaction graph is a tree. If the encoding family $\mathcal H$ is chosen as the complete Reflected encoding, then $\mathrm{OPT}(\mathcal H)=\mathrm{OPT}$.
\end{corollary}

\subsection{Generalized strategy}
\label{subsec:generalized-strategy}

The Generalized strategy generalizes all the above strategies. Let $S^+,S^-\subseteq Q^e$ be disjoint sets. Coordinates in $S^+$ get a monomial term, and coordinates in $S^-$ get a reflected term, and all other coordinates are left free. Thus, the encoding is $\phi_{S^+,S^-}^{\mathrm{G},e}(\mathbf{x}^{e,i}):=\prod_{p\in S^+}x_p^{e,i}\prod_{p\in S^-}(1-x_p^{e,i})$, with the convention that $\phi_{\emptyset,\emptyset}^{\mathrm{G},e}\equiv 1$. This family contains the Monomial strategy when $S^-=\emptyset$ and the Reflected strategy when $S^+=\emptyset$. If $S^+\cup S^-=Q^e$, it recovers a full signature indicator of the Vertex strategy. Since these encoding functions generalize all the previous families, we obtain the following:
\begin{corollary}
\label{cor:generalized-strong}
Suppose that the block interaction graph is a tree. If the encoding family $\mathcal H$ is chosen as the complete Generalized encoding, then $\mathrm{OPT}(\mathcal H)=\mathrm{OPT}$.
\end{corollary}

A direct separation model uses binary variables $z_p^{e,+}$ and $z_p^{e,-}$ to indicate whether coordinate $p$ is fixed to one or to zero. For each current column, introduce binary variables $u_k^{e,i}$ and $u_k^{e,j}$ indicating whether the partial assignment is satisfied. The model maximizing the violation from endpoint $i$ to endpoint $j$ is
\begin{align*}
    \max\quad
        & \sum_{k\in\mathcal K^i}\lambda_{ik}u_k^{e,i}
          -\sum_{k\in\mathcal K^j}\lambda_{jk}u_k^{e,j} \label{eq:g-sep-obj}\\
    \text{s.t.}\quad
        & u_k^{e,l}\le x_p^{e,l,k}+1-z_p^{e,+},
            && l\in\{i,j\},\ k\in\mathcal K^l,\ p\in Q^e,\\
        & \sum_{p\in Q^e}x_p^{e,l,k}z_p^{e,+}
          +\sum_{p\in Q^e}(1-x_p^{e,l,k})z_p^{e,-}
          \le
          \sum_{p\in Q^e}(z_p^{e,+}+z_p^{e,-})-1+u_k^{e,l},
            && l\in\{i,j\},\ k\in\mathcal K^l,\\
        & z_p^{e,+}+z_p^{e,-}\le 1, &&p\in Q^e,\\
        & z_p^{e,+},z_p^{e,-}\in\B,
            && p\in Q^e,\\
        & u_k^{e,l}\in\B,
            && l\in\{i,j\},\ k\in\mathcal K^l.
\end{align*}
As in the Monomial strategy, the sign-reversed separation problem is also solved.

For pricing, once a partial pattern $(S^+,S^-)$ has been selected, introduce a binary variable $w_{S^+,S^-}^{e,i}$ and impose
\begin{align*}
    w_{S^+,S^-}^{e,i} &\le x_p^{e,i},
        && p\in S^+, \\
    w_{S^+,S^-}^{e,i} &\le 1-x_p^{e,i},
        && p\in S^-, \\
    w_{S^+,S^-}^{e,i}
        &\ge
        \sum_{p\in S^+}x_p^{e,i}
        +\sum_{p\in S^-}(1-x_p^{e,i})
        -|S^+|-|S^-|+1. \\
\end{align*}

\section{Comparison of encoding rules}
\label{sec:encoding-comparison}

Recall the four encoding families on the shared variables $\mathbf x\in\{0,1\}^{Q}$ of an edge; we suppress the edge superscript here for simplicity. For disjoint $S^{+},S^{-}\subseteq Q$, the \emph{Generalized} encoding function is $\phi^{G}_{S^{+},S^{-}}(\mathbf x)\;:=\;\prod_{p\in S^{+}}x_{p}\,\prod_{p\in S^{-}}(1-x_{p})$, and the remaining families are its special cases: the \emph{Monomial} function $\phi_{S}(\mathbf x):=\prod_{p\in S}x_{p}$ has $S^{-}=\emptyset$, the \emph{Reflected} function $\phi^{R}_{S}(\mathbf x):=\prod_{p\in S}(1-x_{p})$ has $S^{+}=\emptyset$, and the \emph{Vertex} function, which is a point indicator $\phi_{\mathbf s}(\mathbf x):=\mathbf 1\{\mathbf x= \mathbf s\}$, has $S^{+}\cup S^{-}=Q$. Thus as discussed before, every Vertex, Monomial, or Reflected encoding function is a Generalized one, and the Generalized family contains the other three.

It is therefore straightforward to see that the above containment is inherited by the row counts of column-and-row generation.

\begin{proposition}\label{prop:gen-dominates}
On any instance, the least number of Generalized encoding rows that closes the gap is at most the least number of Vertex, of Monomial, and of Reflected encoding rows that closes the gap.
\end{proposition}

Proposition~\ref{prop:gen-dominates} leaves two questions. Are the three sub-families themselves ordered, or can each beat another? And is the domination of the Generalized family strict, or is the Generalized family efficient on every tractable block? This section answers both, through the results stated next. In each, the instance is a single oriented edge whose two endpoint blocks admit exact linear-optimization oracles of size polynomial in the number of boundary variables and whose coupled set is the singleton $\{(0,\mathbf 0)\}$, so that $\mathrm{OPT}=0$; we call such an instance \emph{tractable}. The constructions and proofs occupy Sections~\ref{sec:mono-refl}--\ref{sec:gen-not-universal} and rest on the gadget of Section~\ref{sec:gadget}.

\begin{theorem}[Monomial versus Reflected]\label{thm:mono-vs-refl}
For every $m\geq 4$ even, with $n=4m^{3}$, there is a tractable instance on $n+1$ boundary variables on which the Reflected family closes the gap with $m=\left(n/4\right)^{1/3}$ rows, while every set of Monomial rows that closes the gap has cardinality at least $\tfrac12\,2^{m/2}=2^{\Omega(n^{1/3})}$.
\end{theorem}

By the complementation symmetry of the gadget the two roles reverse.

\begin{corollary}[Reflected versus Monomial]\label{cor:refl-vs-mono}
For every $m\geq 4$ even, with $n=4m^{3}$, there is a tractable instance on $n+1$ boundary variables on which the Monomial family closes the gap with $m=\left(n/4\right)^{1/3}$ rows, while every set of Reflected rows that closes the gap has cardinality at least $\tfrac12\,2^{m/2}=2^{\Omega(n^{1/3})}$.
\end{corollary}

With Proposition~\ref{prop:gen-dominates}, these two place the Generalized family strictly above both sub-families.

\begin{theorem}[Generalized strictly dominates Monomial and Reflected]\label{thm:gen-strict-mono-refl}
For every $m\geq 4$ even, with $n=4m^{3}$, there is a tractable instance on which the Monomial family requires $2^{\Omega(n^{1/3})}$ rows while the Generalized family closes the gap with $m=\left(n/4\right)^{1/3}$ rows, and a tractable instance on which the Reflected family requires $2^{\Omega(n^{1/3})}$ rows while the Generalized family closes the gap with $m=\left(n/4\right)^{1/3}$ rows.
\end{theorem}

The Vertex family can also be dominated as sharply.

\begin{theorem}[Generalized strictly dominates Vertex]\label{thm:vertex-vs-gen}
For every $n\geq 6$ divisible by 3, there is a tractable instance on $n+1$ boundary variables on which the Generalized family closes the gap with two rows, while every set of Vertex rows that closes the gap has cardinality at least $2^{n/3}=2^{\Omega(n)}$.
\end{theorem}

Finally, the Generalized family, richest of the four, is itself not universal: it can be exponential on a tractable instance that a single encoding outside the family closes.

\begin{theorem}[Generalized lower bound]\label{thm:generalized-encoding-lower-bound}
For every $n\geq 12$ there is a tractable instance on $n+1$ boundary variables, the \emph{parity instance}, such that every set of Generalized rows that closes the gap has cardinality at least $\tfrac12(4/3)^{\lceil n/4\rceil}=2^{\Omega(n)}$.
\end{theorem}

\begin{proposition}[A single non-family encoding closes the gap]\label{prop:parity-upper-bound}
For the parity instance of Theorem~\ref{thm:generalized-encoding-lower-bound}, there is a nonlinear encoding outside the Generalized family such that adding the single consistency row closes the gap.
\end{proposition}

\begin{remark}
\label{rem:modeling}
The examples used in the proofs of this section also suggest a practical modeling principle. Whenever the coupling equalities imply simple restrictions involving variables that belong to different blocks, it is useful to translate those implications into local constraints before running column-and-row generation. Such constraints are redundant for the complete coupled formulation, but they may be nonredundant after local convexification and may substantially strengthen the pricing problems. In this sense, part of the modeling effort should be to expose as many cross-block implications as possible through local block descriptions, rather than leaving all of them to be discovered by the generated edge encodings.
\end{remark}

\subsection{A common gadget}
\label{sec:gadget}

All four separations live on one single-edge gadget, parametrized by a subset of the cube. Let the $\mathbf z$-coordinates be indexed by a finite set $R$, take an oriented edge $e=ij$ with boundary $Q=\{0\}\cup R$, and write a boundary signature as $\mathbf x:=(\tau,\mathbf z)$ with $\tau=s_{0}$ and $z_{p}=s_{p}$ for $p\in R$. Fix $F\subseteq\{0,1\}^{R}$ and write $\bar F:=\{0,1\}^{R}\setminus F$. The two endpoint blocks are $X^{i}:=\{(0,\mathbf 0)\}\cup\{(1,\mathbf z):\mathbf z\in F\}$ and $X^{j}:=\{(0,\mathbf 0)\}\cup\{(1,\mathbf z):\mathbf z\in\bar F\}$, with local costs $f_{i}(\mathbf{x}^{i})=-\tau$ and $f_{j}(\mathbf{x}^{j})=-\tau$. We call this the \emph{$F$-gadget}. Its value and its row requirement are completely captured by the following separation problem.

\begin{lemma}[Gadget lemma]\label{lem:gadget}
For the $F$-gadget the coupled set is the singleton $\{(0,\mathbf 0)\}$ and $\mathrm{OPT}=0$. Moreover, for any encoding family $\mathcal H$ and any positive integer $k$, the following are equivalent:
\begin{enumerate}
\item[(a)] there is a set $\mathcal C$ of $k$ rows from $\mathcal H$ whose augmented Dantzig-Wolfe master certifies $\mathrm{OPT}=0$ with no negative reduced-cost column;
\item[(b)] there are encoding functions $\phi_{1},\dots,\phi_{k}\in\mathcal H$ and reals $\pi_{0},(\pi_{p})_{p\in R},(\mu_{h})_{h}$ such that
  $
    r(\mathbf z)\;:=\;\Big(\pi_{0}-\sum_{h=1}^{k}\mu_{h}\,\phi_{h}(0,\mathbf 0)\Big)
    +\sum_{p\in R}\pi_{p}z_{p}+\sum_{h=1}^{k}\mu_{h}\,\phi_{h}(1,\mathbf z)
  $
satisfies
  \begin{equation}\label{eq:margin}
    r(\mathbf z)\leq -1 \ \ (\mathbf z\in F),\qquad r(\mathbf z)\geq 1 \ \ (\mathbf z\in\bar F).
  \end{equation}
\end{enumerate}
We call any $r$ satisfying \eqref{eq:margin} a \emph{margin separator for $F$}.
\end{lemma}

\begin{proof}
After imposing $\mathbf{x}^{i}=\mathbf{x}^{j}$ a point with $\tau=1$ would require $\mathbf z\in F$ and $\mathbf z\in\bar F$ simultaneously, so $X^{i}\cap X^{j}=\{(0,\mathbf 0)\}$, whence $\mathrm{OPT}=0$.

Let $\mathcal C$ be a finite set of encoding rows. Let $\alpha_{i},\alpha_{j}$ be the dual multipliers of the two convexity constraints, $\pi_{p}$ ($p\in Q$) those of the original coupling rows, and $\mu_{h}$ ($h\in\mathcal C$) those of the generated rows. Put
$
  L(\tau,\mathbf z):=\pi_{0}\tau+\sum_{p\in R}\pi_{p}z_{p}+\sum_{h\in\mathcal C}
  \mu_{h}\,\phi_{h}(\tau,\mathbf z).
$
Because the edge is oriented from $i$ to $j$, the reduced cost of a candidate column is
$\bar c_{i}(\tau,\mathbf z)=-\tau-\alpha_{i}-L(\tau,\mathbf z)$ over $X^{i}$ and $\bar c_{j}(\tau,\mathbf z)=-\tau-\alpha_{j}+L(\tau,\mathbf z)$ over $X^{j}$.

\emph{(a)$\Rightarrow$(b).} The only rows with nonzero right-hand side are the two convexity constraints, each equal to $1$, so a dual solution of objective value $\mathrm{OPT}=0$ has $\alpha_{i}+\alpha_{j}=0$. Certification gives $\bar c_{i},\bar c_{j}\geq 0$ throughout; at the origin $\bar c_{i}(0,\mathbf 0)=-\alpha_{i}-L(0,\mathbf 0)\geq 0$ and $\bar c_{j}(0,\mathbf 0)=-\alpha_{j}+L(0,\mathbf 0)\geq 0$, and adding them with $\alpha_{i}+\alpha_{j}=0$ forces equality, so $-\alpha_{i}=\alpha_{j}=L(0,\mathbf 0)$. For $\mathbf z\in F$, $(1,\mathbf z)\in X^{i}$ yields $\bar c_{i}(1,\mathbf z)=-1+L(0,\mathbf 0)-L(1,\mathbf z)\geq 0$; for $\mathbf z\in\bar F$, $(1,\mathbf z)\in X^{j}$ yields $\bar c_{j}(1,\mathbf z)=-1-L(0,\mathbf 0)+L(1,\mathbf z)\geq 0$. Hence $r(\mathbf z):=L(1,\mathbf z)-L(0,\mathbf 0)$ satisfies \eqref{eq:margin}. Substituting $L$ and cancelling the constant gives the displayed form.

\emph{(b)$\Rightarrow$(a).} Given $r$ as in (b), read off duals: take $\mu_{h}$, $\pi_{p}$, and $\pi_{0}$ as written, and set $\alpha_{i}=-L(0,\mathbf 0)$, $\alpha_{j}=L(0,\mathbf 0)$, so that $\alpha_{i}+\alpha_{j}=0$ and the dual objective is $0$. Then $\bar c_{i}(0,\mathbf 0)=\bar c_{j}(0,\mathbf 0)=0$, while for $\mathbf z\in F$, $\bar c_{i}(1,\mathbf z)=-1-r(\mathbf z)\geq 0$ by \eqref{eq:margin}, and for $\mathbf z\in\bar F$, $\bar c_{j}(1,\mathbf z)=-1+r(\mathbf z)\geq 0$. No column has negative reduced cost, so the rows certify $\mathrm{OPT}=0$.
\end{proof}

Lemma~\ref{lem:gadget} reduces every claim below to a separation statement: an \emph{upper bound} for a family is a short margin separator built from its encoding functions, and a \emph{lower bound} is a proof that any margin separator built from its functions needs many of them.

We also record a complementation symmetry. Let $\kappa(\mathbf z):=\mathbf 1-\mathbf z$ denote coordinatewise complementation on $\{0,1\}^{R}$. Then $\mathbf z\in\kappa(F)$ if and only if $\kappa(\mathbf z)\in F$, so the substitution $r\mapsto r\circ\kappa$ is a bijection between margin separators for $F$ and margin separators for $\kappa(F)$ that preserves the number of encoding terms. On $\{0,1\}^{R}$, composition with $\kappa$ maps affine functions to affine functions, exchanges the Monomial and Reflected functions ($\phi^M_{S}\circ\kappa=\phi^{R}_{S}$ and $\phi^{R}_{S}\circ\kappa=\phi^M_{S}$), maps $\phi^{G}_{S^{+},S^{-}}$ to $\phi^{G}_{S^{-},S^{+}}$, and maps the Vertex indicator of $s$ to that of $\kappa(s)$. Therefore, for each family, the least number of terms from the family in a margin separator for $F$ equals the least number of terms from the complemented family in a margin separator for $\kappa(F)$; combined with Lemma~\ref{lem:gadget}, this transfers row bounds between the two gadgets, with the Monomial and Reflected roles exchanged. Finally, since $\min\{\langle c,\mathbf z\rangle:\mathbf z\in\kappa(F)\} =\langle c,\mathbf 1\rangle-\max\{\langle c,\mathbf z\rangle:\mathbf z\in F\}$, linear-optimization oracles transfer under $\kappa$, so the $\kappa(F)$-gadget is tractable whenever the $F$-gadget is.

\subsection{Monomial versus Reflected versus Generalized}
\label{sec:mono-refl}

Fix $m\geq 4$ even. Set $w:=4m^{2}$ and $n:=mw=4m^{3}$, and index the $\mathbf z$-coordinates by $m$ disjoint groups $G_{1},\dots,G_{m}$ of size $w$, writing $z_{ab}$ for the $b$-th coordinate of group $a$ ($a\in[m]$, $b\in[w]$). Consider the set-covering system
$\sum_{b=1}^{w} z_{ab}\geq 1$ for $a\in[m]$, where $\mathbf z\in\{0,1\}^{n}$,
and write $\mathrm{MP}(\mathbf z)=1$ when $\mathbf z$ satisfies every covering inequality and $\mathrm{MP}(\mathbf z)=0$ otherwise; equivalently, $\mathrm{MP}(\mathbf z)=0$ exactly when some group is uncovered, $z_{G_{a}}=\mathbf 0$. The \emph{covering instance} is the $F$-gadget of Section~\ref{sec:gadget} with $F=\{\mathbf z:\mathrm{MP}(\mathbf z)=1\}$, the integer points of the covering system.

Write $\deg_{\pm}(g)$ for the \emph{polynomial threshold degree} of a $\{0,1\}$-valued function $g$, the least degree of a polynomial $p$ with $p(\mathbf z)>0$ where $g(\mathbf z)=1$ and $p(\mathbf z)<0$ where $g(\mathbf z)=0$. Let $\mathrm{MP}_{k}$ denote the group-covering function on $k$ groups of $4k^{2}$ variables, that is $\mathrm{MP}_{k}(\mathbf y)=1$ exactly when $\sum_{b=1}^{4k^{2}} y_{ab}\geq 1$ for every $a\in[k]$. We use the following bound (Minsky and Papert~\cite{minsky-papert}).

\begin{theorem}[Minsky--Papert, Section 3.2~\cite{minsky-papert}]\label{thm:mp-degree}
For every $k\geq 1$, $\deg_{\pm}(\mathrm{MP}_{k})\geq k$.
\end{theorem}

We also use that $\deg_{\pm}$ is monotone under coordinate fixing: if $g$ is obtained from $f$ by fixing some variables, then $\deg_{\pm}(g)\leq\deg_{\pm}(f)$, since any polynomial attaining $\deg_{\pm}(f)$ restricts to one for $g$ of no larger degree.

The Monomial lower bound rests on the density estimate of Lemma~\ref{lem:mono-density} below, the Monomial-basis counterpart of Lemma~\ref{lem:parity-rows}. 
Both lower bounds use the following binomial tail estimate: for every integer $N\geq 1$,
\begin{equation}\label{eq:binomial-tail}
\sum_{t\leq N/4}\binom{N}{t}
\;\leq\;
4^{N/4}\bigl(\tfrac{4}{3}\bigr)^{3N/4}
\;=\;
2^{N}\bigl(\tfrac{16}{27}\bigr)^{N/4}.
\end{equation}
Indeed, each term of the expansion $1=(\tfrac14+\tfrac34)^{N}$ with $t\leq N/4$ satisfies $(\tfrac14)^{t}(\tfrac34)^{N-t}\geq(\tfrac14)^{N/4}(\tfrac34)^{3N/4}$, so keeping only these terms and dividing gives \eqref{eq:binomial-tail}.

\begin{lemma}\label{lem:mono-density}
Let $r(\mathbf z)=l(\mathbf z)+\sum_{h=1}^{D}\mu_{h}\,\phi_{h}(\mathbf z)$, where each $\phi_{h}(\mathbf z)=\prod_{p\in S_{h}}z_{p}$ is a Monomial encoding function and $l$ has degree at most one. Suppose $r(\mathbf z)\leq -1$ whenever $\mathrm{MP}(\mathbf z)=1$ and $r(\mathbf z)\geq 1$ whenever $\mathrm{MP}(\mathbf z)=0$. Then $D\geq\tfrac12\,2^{m/2}=2^{\Omega(n^{1/3})}$.
\end{lemma}

\begin{proof}
Set $d:=m/2$ and suppose, for contradiction, that $D<\tfrac12\,2^{d}$. Let $\mathcal R$ be the set of coordinate fixings that assign each variable $z_{ab}$ one of two states, \emph{free} or \emph{fixed to $0$}; then $|\mathcal R|=2^{mw}$. For $\rho\in\mathcal R$ and any function $f$ on $\{0,1\}^{n}$, let $f|_{\rho}$ denote its restriction to the face of $\rho$: the function of the free variables obtained by substituting the fixed values, so that $f|_{\rho}(\mathbf z)=f(\mathbf z')$ whenever $\mathbf z'$ completes $\rho$ with the free values $\mathbf z$. Let $F_{a}$ be the index set of free variables of group $a$. In particular, $r|_{\rho}$ is a polynomial in the free variables and, since $\rho$ fixes variables to $0$ only, $\mathrm{MP}|_{\rho}$ is again a group-covering indicator: $\mathrm{MP}|_{\rho}(\mathbf z)=1$ exactly when $\sum_{b\in F_{a}}z_{ab}\geq 1$ for every $a\in[m]$.

A term $\phi_{h}$ with $|S_{h}|\geq d$ becomes the zero term under any $\rho$ that fixes some variable of $S_{h}$ to $0$, and remains a term of degree $|S_{h}|$ only when every variable of $S_{h}$ is free. The number of $\rho\in\mathcal R$ leaving $\phi_{h}$ nonzero is therefore $2^{mw-|S_{h}|}\leq 2^{mw-d}$. Summing over the at most $D$ terms of degree at least $d$, the number of $\rho$ under which any such term survives is at most $D\,2^{mw-d}<\tfrac12\,2^{mw}$.

For a fixed group $a$, the number of $\rho$ with $|F_{a}|<m^{2}$ is $2^{(m-1)w}\sum_{t<m^{2}}\binom{w}{t}$. Since $w=4m^{2}$, the tail estimate~\eqref{eq:binomial-tail} with $N=w$, so that $N/4=m^{2}$, gives $\sum_{t<m^{2}}\binom{w}{t}\leq 2^{w}(16/27)^{m^{2}}$, and the count is therefore at most $2^{mw}(16/27)^{m^{2}}$. Over the $m$ groups, at most $m\,2^{mw}(16/27)^{m^{2}}$ fixings leave some group with fewer than $m^{2}$ free variables.
We have that $m\,(16/27)^{m^{2}}<\tfrac14$ for every $m\geq 2$: the value at $m=2$ is $2\,(16/27)^{4}<\tfrac14$ and the ratio of consecutive values, $\tfrac{m+1}{m}(16/27)^{2m+1}\leq \tfrac{3}{2}(16/27)^5$, is below one. Therefore, the two counts sum to fewer
than $(\tfrac12+\tfrac14)\,2^{mw}<2^{mw}$. Hence some $\rho\in\mathcal R$ retains no term of degree at least $d$ and gives every group at least $m^{2}$ free variables. Fix such a $\rho$.

We claim that $\mathrm{MP}_{m/2}$ is a fixing of $\mathrm{MP}|_{\rho}$. Define a fixing $\sigma$ of the free variables as follows: in each group $a>m/2$, set one free variable to $1$ and the remaining free variables of that group to $0$; in each group $a\leq m/2$, fix all but $4(m/2)^{2}=m^{2}$ of the free variables to $0$, which is possible since $|F_{a}|\geq m^{2}$. Under $\sigma$, every group $a>m/2$ is covered outright, while a group $a\leq m/2$ is covered exactly when one of its $m^{2}$ surviving variables equals $1$. Hence $(\mathrm{MP}|_{\rho})|_{\sigma}=\mathrm{MP}_{m/2}$, the group-covering function on $m/2$ groups of $4(m/2)^{2}$ variables. By Theorem~\ref{thm:mp-degree}, $\deg_{\pm}(\mathrm{MP}_{m/2})\geq m/2$, and since $\deg_{\pm}$ is monotone under fixing,
$
\deg_{\pm}(\mathrm{MP}|_{\rho})
\;\geq\;
\deg_{\pm}\!\big((\mathrm{MP}|_{\rho})|_{\sigma}\big)
\;=\;
\deg_{\pm}(\mathrm{MP}_{m/2})
\;\geq\; m/2.
$

Every completion of $\rho$ on the free variables is a point of $\{0,1\}^{n}$, so the hypotheses give $r|_{\rho}(\mathbf z)\leq -1$ where $\mathrm{MP}|_{\rho}(\mathbf z)=1$ and $r|_{\rho}(\mathbf z)\geq 1$ where $\mathrm{MP}|_{\rho}(\mathbf z)=0$. Thus $-r|_{\rho}$ is positive where $\mathrm{MP}|_{\rho}=1$ and negative where $\mathrm{MP}|_{\rho}=0$, so $\deg(r|_{\rho})\geq\deg_{\pm}(\mathrm{MP}|_{\rho})\geq m/2$. But $r|_{\rho}$ retains no term of degree at least $d$, and $l$ has degree at most one, so $\deg(r|_{\rho})\leq d-1<m/2$, a contradiction. Therefore $D\geq\tfrac12\,2^{d}=\tfrac12\,2^{m/2}=2^{\Omega(n^{1/3})}$.
\end{proof}

\begin{proof}[Proof of Theorem~\ref{thm:mono-vs-refl}]
\emph{Tractability.} Consider a linear objective on $X^{i}$. On the $\tau=1$ part the constraints are the covering inequalities $\sum_{b=1}^{w} z_{ab}\geq 1$, one per group; since the groups occupy disjoint coordinates, the objective separates across groups and each group is optimized in $O(w)$ time by selecting its negative-cost variables and, if none is selected, the single variable of least cost. Comparing with the value of the origin gives the optimum in $O(n)$ time. On $X^{j}$ the $\tau=1$ part is $\bigcup_{a=1}^{m}\{\mathbf z:z_{G_{a}}=\mathbf 0\}$, a union of $m$ subcubes; optimizing over each and comparing with the origin gives the optimum in $O(mn)$ time. Both oracles are polynomial in $n$. Finally, by Lemma~\ref{lem:gadget}, the coupled set is the singleton $\{(0,\mathbf 0)\}$, so $\mathrm{OPT}=0$ and the coupled set is trivial to optimize over.

\emph{Reflected upper bound.} For each group $a\in[m]$ take the Reflected encoding function $\phi^{R}_{G_{a}}(\mathbf z)=\prod_{b=1}^{w}(1-z_{ab})$, the indicator that group $a$ is uncovered (all of its variables are $0$). Then $\sum_{a}\phi^{R}_{G_{a}}(\mathbf z)$ counts the uncovered groups, so it is $0$ when $\mathbf z$ is feasible ($\mathrm{MP}(\mathbf z)=1$) and at least $1$ when $\mathbf z$ is infeasible ($\mathrm{MP}(\mathbf z)=0$). Hence $r(\mathbf z):=-1+2\sum_{a=1}^{m}\phi^{R}_{G_{a}}(\mathbf z)$ is a margin separator for $F$ built from $m$ Reflected functions, and by Lemma~\ref{lem:gadget} these $m=(n/4)^{1/3}$ rows close the gap.

\emph{Monomial lower bound.} By Lemma~\ref{lem:gadget}, a finite set $\mathcal C$ of Monomial rows certifies $\mathrm{OPT}=0$ if and only if there is a margin separator for $F$, namely an affine function of $\mathbf z$ plus a $\mu$-combination of the functions $\phi_{h}(1,\cdot)$, $h\in\mathcal C$, satisfying \eqref{eq:margin}, that is $r(\mathbf z)\leq -1$ where $\mathrm{MP}(\mathbf z)=1$ and $r(\mathbf z)\geq 1$ where $\mathrm{MP}(\mathbf z)=0$. Fixing $\tau=1$ in a Monomial encoding function yields a Monomial function of $\mathbf z$ or a constant absorbed into the affine part, so $r$ has the form of Lemma~\ref{lem:mono-density} with at most $|\mathcal C|$ Monomial terms, and the displayed conditions are its hypotheses. Lemma~\ref{lem:mono-density} gives $|\mathcal C|\geq\tfrac12\,2^{m/2}=2^{\Omega(n^{1/3})}$.
\end{proof}

\begin{proof}[Proof of Corollary~\ref{cor:refl-vs-mono}]
Let $F$ be the feasible set of the covering instance of Theorem~\ref{thm:mono-vs-refl} and consider the $\kappa(F)$-gadget, where $\kappa(F)$ is the integer point set of the complemented system $\sum_{b=1}^{w}(1-z_{ab})\geq 1$, equivalently $\sum_{b=1}^{w}z_{ab}\leq w-1$, $a\in[m]$ (no group is fully selected). By the complementation symmetry of Section~\ref{sec:gadget}, this gadget is tractable and its coupled set is the singleton $\{(0,\mathbf 0)\}$, and $k$ rows from a family close its gap if and only if $k$ rows from the complemented family close the gap of the covering instance. Since $\kappa$ exchanges the Monomial and Reflected functions, the Reflected upper bound of Theorem~\ref{thm:mono-vs-refl} yields a Monomial upper bound of $(n/4)^{1/3}$ rows for the $\kappa(F)$-gadget, and its Monomial lower bound yields a Reflected lower bound of $\tfrac12\,2^{m/2}$.
\end{proof}

\begin{proof}[Proof of Theorem~\ref{thm:gen-strict-mono-refl}]
On the covering instance the $m$ Reflected functions of the upper bound are Generalized functions, so by Proposition~\ref{prop:gen-dominates} the Generalized family closes the gap with $m=(n/4)^{1/3}$ rows, whereas Theorem~\ref{thm:mono-vs-refl} forces $2^{\Omega(n^{1/3})}$ Monomial rows. On the instance of Corollary~\ref{cor:refl-vs-mono} the $m$ Monomial functions are likewise Generalized functions, so the Generalized family closes the gap with $(n/4)^{1/3}$ rows, whereas the Reflected family requires $2^{\Omega(n^{1/3})}$.
\end{proof}

\subsection{Vertex versus Generalized}
\label{sec:vertex-gen}
Fix $n\geq 6$ divisible by 3, let $R=\{1,\dots,n\}$, and partition $R$ into three equal parts $S^{+},S^{-},W$ of size $n/3$. Define the following special faces of the cube: $C_{1}:=\{\mathbf z:z_{S^{+}}=\mathbf 1,\ z_{S^{-}}=\mathbf 0\}$, $C_{2}:=\{\mathbf z:z_{S^{+}}=\mathbf 0,\ z_{S^{-}}=\mathbf 1\}$,
each with the coordinates in $W$ free, so $|C_{1}|=|C_{2}|=2^{n/3}$. This \emph{opposing-face instance} is the $F$-gadget with $F=C_{1}\cup C_{2}$.
\begin{proof}[Proof of Theorem~\ref{thm:vertex-vs-gen}]
\emph{Tractability.} Over the $\tau=1$ part of $X^{i}$ optimize a linear objective over $C_{1}$ and over $C_{2}$ separately, each a subcube, and take the better; over $X^{j}$ optimize over the full cube and, if the optimum lands in $C_{1}\cup C_{2}$, flip the cheapest coordinate that leaves $C_{1}\cup C_{2}$. Comparing with the origin gives both oracles in time $O(n)$.

\emph{Generalized upper bound.} The functions $\phi^{G}_{S^{+},S^{-}}$ and $\phi^{G}_{S^{-},S^{+}}$ are the indicators of $C_{1}$ and $C_{2}$, each a Generalized encoding function carrying $n/3$ uncomplemented and $n/3$ complemented coordinates. Since $C_{1}$ and $C_{2}$ are disjoint, $r(\mathbf z):=1-2\phi^{G}_{S^{+},S^{-}}(\mathbf z)-2\phi^{G}_{S^{-},S^{+}}(\mathbf z)$ equals $-1$ on $C_{1}\cup C_{2}=F$ and $+1$ on $\bar F$, hence is a margin separator with two Generalized functions, and Lemma~\ref{lem:gadget} closes the gap with these two rows. Neither function is Monomial, Reflected, or Vertex: each mixes uncomplemented and complemented coordinates and has support $2n/3<n$.

\emph{Vertex lower bound.} We show that closing the gap with Vertex rows needs at least $2^{n/3}$ of them.

A Vertex row uses a point indicator $\phi_{\mathbf s}(\mathbf z)=\mathbf 1\{\mathbf z=\mathbf s\}$, which is $1$ at the single signature $\mathbf s$ and $0$ everywhere else.

By Lemma~\ref{lem:gadget}, if a set $\mathcal C$ of Vertex rows closes the gap, there is a margin separator $r(\mathbf z)=\ell(\mathbf z)+\sum_{h}\mu_{h}\,\phi_{h}(1,\mathbf z)$ with $\ell$ affine. The restriction at $\tau=1$ of a Vertex indicator $\phi_{\mathbf s}$, $\mathbf s\in\{0,1\}^{Q}$, is $\mathbf 1\{\mathbf z=\mathbf s_{R}\}$ if $s_{0}=1$ and the constant $0$ if $s_{0}=0$. Hence, $r$ can be written as $r(\mathbf z)=\ell(\mathbf z)+\sum_{h}\mu_{h}\,\mathbf 1\{\mathbf z=\mathbf s_{h}\}$, and satisfies $r(\mathbf z)\leq -1$ for $\mathbf z\in F$ and $r(\mathbf z)\geq 1$ for $\mathbf z\in\bar F$.

Let $E=\{\mathbf s_{h}\}$ be the set of signatures used by these indicators, so $|E|\leq|\mathcal C|$. At any point $\mathbf z\notin E$ all indicators vanish, so $r(\mathbf z)=\ell(\mathbf z)$ there. The requirements on $r$ therefore become requirements on the affine function alone: $\ell(\mathbf z)\leq -1$ for $\mathbf z\in F\setminus E$ and $\ell(\mathbf z)\geq 1$ for $\mathbf z\in\bar F\setminus E$. That is, $\ell$ must be negative on the two tails $C_{1},C_{2}$ and positive in the middle $\bar F$, except at the points of $E$, which the indicators are free to correct. It thus suffices to show that any affine $\ell$ violates these sign requirements at no fewer than $2^{n/3}$ points; each such point must lie in $E$, giving $|\mathcal C|\geq|E|\geq 2^{n/3}$.

Fix the coordinates in $W$ to some value $\mathbf y\in\{0,1\}^{W}$, and consider the four signatures that share this $\mathbf y$: $ \mathbf u_{y}:=(\mathbf 1_{S^{+}},\mathbf 0_{S^{-}},\mathbf y),\ \mathbf v_{y}:=(\mathbf 0_{S^{+}},\mathbf 1_{S^{-}},\mathbf y),\ \mathbf m_{y}:=(\mathbf 1_{S^{+}},\mathbf 1_{S^{-}},\mathbf y),\ \mathbf m'_{y}:=(\mathbf 0_{S^{+}},\mathbf 0_{S^{-}},\mathbf y). $ Here $\mathbf u_{y}\in C_{1}$ and $\mathbf v_{y}\in C_{2}$, so both lie in $F$ and $\ell$ should be $\leq -1$ on them; while $\mathbf m_{y},\mathbf m'_{y}\in\bar F$, so $\ell$ should be $\geq 1$ on them.

Coordinate by coordinate, $\mathbf u_{y}+\mathbf v_{y}=\mathbf m_{y}+\mathbf m'_{y}$. Since $\ell$ is affine, we have $\ell(\mathbf u_{y})+\ell(\mathbf v_{y})=\ell(\mathbf m_{y})+\ell(\mathbf m'_{y})$. If $\ell$ had the correct sign at all four points, the left side would be at most $-2$ and the right side at least $+2$, which is impossible. Hence $\ell$ has the wrong sign at one of the four points, and that point lies in $E$.

The four points all carry the same $W$-block $\mathbf y$, so the quadruples for different $\mathbf y$ are disjoint. There are $2^{n/3}$ choices of $\mathbf y$, each forcing a distinct point into $E$. Therefore $|E|\geq 2^{n/3}$, and so $|\mathcal C|\geq 2^{n/3}$.
\end{proof}

\subsection{The Generalized family is not universal}
\label{sec:gen-not-universal}

We begin by introducing some definitions and the tool used in the proof. A real polynomial $p$ \emph{sign-represents} a function $g\colon\B^{n}\to\{-1,+1\}$ if $\mathrm{sign}\,p(\mathbf x)=g(\mathbf x)$ for every $\mathbf x\in\B^{n}$. The \emph{sign degree} of $g$ is the minimum degree of any polynomial that sign-represents $g$.

\begin{theorem}[Minsky--Papert, Section 3.1~\cite{minsky-papert}]\label{thm:minsky-papert-parity}
The sign degree of the parity function $\mathbf x\mapsto-(-1)^{|\mathbf x|}$ on $n\geq 1$ variables equals $n$, where $|\mathbf x|:=\sum_{q=1}^{n}x_{q}$ denotes the Hamming weight of $\mathbf x\in\B^{n}$.
\end{theorem}

\begin{lemma}\label{lem:parity-rows}
Let $n\geq 12$ and let $ r(\mathbf{z})=l(\mathbf{z})+\sum_{h=1}^{m}\mu_h\phi_h(\mathbf{z}), $ where each $\phi_h$ is a Generalized encoding function and $l$ has degree at most one. Suppose $r(\mathbf{z})\leq -1$ for $|\mathbf{z}|$ even and $r(\mathbf{z})\geq 1$ for $|\mathbf{z}|$ odd. Then $m\geq\tfrac12(4/3)^{\lceil n/4\rceil}=2^{\Omega(n)}$.
\end{lemma}

\begin{proof}[Proof of Lemma~\ref{lem:parity-rows}]
The hypothesis says exactly that $\textup{sign}\, r(\mathbf{z})=-(-1)^{|\mathbf{z}|}$ with $|r(\mathbf{z})|\geq 1$ for all $\mathbf{z}\in\{0,1\}^n$; in particular, $r$ sign-represents parity. Let $S_h\subseteq\{1,\dots,n\}$ be the support of $\phi_h$ (i.e., the set of coordinates $S^+ \cup S^-$ in its definition), so that $\deg\phi_h=|S_h|$.

Set $d=\lceil n/4\rceil\geq 2$. Let $m_{\geq d}$ denote the number of rows $\phi_h$ with $|S_h|\geq d$. We will prove that $m_{\geq d} \geq \tfrac12(4/3)^d$. Since $m \geq m_{\geq d}$, this proves the result.

Recall that each row $\phi_h$ may be written as $\phi_h(\mathbf z)=\prod_{p\in S_h}\mathbf{1}\{z_p=a_p\}$ for $a_p\in\{0,1\}$. We say a fixed coordinate $p$ \emph{annihilates} $\phi_h$ if its assigned value forces this product to be identically $0$, that is, if the factor $\mathbf{1}\{z_p=a_p\}$ becomes the constant $0$.

Consider the set $\mathcal R$ of coordinate labelings that assign to each coordinate $p\in\{1,\dots,n\}$ one of four labels, $\mathrm{free}_{0}$, $\mathrm{free}_{1}$, $\mathrm{fix}_0$, or $\mathrm{fix}_1$, so that $|\mathcal R|=4^{n}$. A labeling $\rho$ acts as a restriction: coordinates labeled $\mathrm{free}_{0}$ or $\mathrm{free}_{1}$ remain free, and the others are fixed to the indicated bit; the two free labels are duplicates whose only purpose is to make the four labels equally numerous. The fixed
coordinates become constants, $r|_{\rho}$ is a polynomial in the $n'=n'(\rho)$ free coordinates, and we call the corresponding subcube the \emph{face}.

Fix a row $\phi_{h}$ with support of size $k\geq d$ and a support coordinate $p$, whose factor is $\mathbf 1\{z_{p}=a_{p}\}$. Among the four labels of $p$, exactly one, namely $\mathrm{fix}_{1-a_{p}}$, annihilates $\phi_{h}$. Since a labeling chooses one label per coordinate, the number of $\rho\in\mathcal R$ under which $\phi_{h}$ is annihilated by none of its $k$ support coordinates is the product of the per-coordinate counts, three at each support coordinate and four elsewhere: $3^{k}4^{n-k}\leq(3/4)^{d}4^{n}$.

Suppose, for contradiction, that $m_{\geq d}<\tfrac12(4/3)^{d}$. Since $(4/3)^{d}(3/4)^{d}=1$, summing the survival counts over the rows with $|S_{h}|\geq d$ shows that fewer than $\tfrac12\,4^{n}$ labelings leave some such row not annihilated.

Next, a labeling with free set of size $t$ is determined by the free set, one of two free labels on each free coordinate, and one of two bits on each fixed coordinate, so the number of $\rho$ with $n'<n/4$ is $2^{n}\sum_{t<n/4}\binom{n}{t}$. The tail  estimate~\eqref{eq:binomial-tail} with $N=n$ gives $\sum_{t<n/4}\binom{n}{t}\leq 2^{n}(16/27)^{n/4}$, so at most $4^{n}(16/27)^{n/4}$ labelings satisfy $n'<n/4$, and $(16/27)^{n/4}\leq(16/27)^{3}<\tfrac14$ for $n\geq 12$. The two counts sum to fewer than $(\tfrac12+\tfrac14)\,4^{n}<4^{n}$ labelings, so some $\rho\in\mathcal R$ annihilates every row with $|S_{h}|\geq d$ and satisfies $n'\geq n/4$. Fix such a $\rho$.

On its face, $r|_\rho$ is the affine part plus rows with $|S_h|<d$, so $\deg(r|_\rho)\leq d-1< n/4\leq n'$. Every assignment to the free coordinates, together with the bits fixed by $\rho$, is a point of $\{0,1\}^{n}$, so each value of $r|_{\rho}$ on the face is a value of $r$; in particular $r|_{\rho}(\mathbf z)\leq -1$ or $r|_{\rho}(\mathbf z)\geq 1$ at every point of the face. Moreover, fixing the assigned coordinates changes the parity of a point only by the constant amount contributed by those bits, so on the face $r|_\rho$ sign-represents the parity of the $n'$ free coordinates, up to a single global sign determined by the assigned bits. By Theorem~\ref{thm:minsky-papert-parity}, sign-representing parity on $n'$ variables requires degree $n'$, contradicting $\deg(r|_\rho)<n'$. Therefore $m\geq m_{\geq d}\geq\tfrac12(4/3)^{\lceil n/4\rceil}=2^{\Omega(n)}$.
\end{proof}

\begin{proof}[Proof of Theorem~\ref{thm:generalized-encoding-lower-bound}]
Let $E_n:=\{\mathbf{z}\in\B^n:\sum_{p=1}^n z_p\text{ is even}\}$ and $O_n:=\{\mathbf{z}\in\B^n:\sum_{p=1}^n z_p\text{ is odd}\}$. The \emph{parity instance} is the $F$-gadget of Section~\ref{sec:gadget} with $F=E_n$ (so $\bar F=O_n$); that is, the edge $e=ij$ carries the two endpoint blocks $X^i := \{(0,\mathbf{0})\} \cup \{(1,\mathbf{z}):\mathbf{z}\in E_n\}$ and $X^j := \{(0,\mathbf{0})\} \cup \{(1,\mathbf{z}):\mathbf{z}\in O_n\}$, with local costs $f_i(\mathbf{x}^{i})=-\tau$ and $f_j(\mathbf{x}^{j})=-\tau$.

\emph{Tractability.} Given a linear objective, optimizing over the $\tau=1$ part of $X^i$ reduces to optimizing over the even-parity vectors; this is done by taking the coordinatewise best vector and, if its parity is wrong, flipping a coordinate of minimum loss, and then comparing with the value of the origin. The same argument applies to $X^j$ with odd parity. Thus, each endpoint block admits an $\mathcal O(n)$ linear-optimization oracle. By Lemma~\ref{lem:gadget}, the coupled set is the singleton $X^i\cap X^j=\{(0,\mathbf{0})\}$, so $\mathrm{OPT}=0$ and the coupled set is trivial to optimize over.

\emph{Generalized lower bound.} By Lemma~\ref{lem:gadget}, a finite set $\mathcal C$ of Generalized encoding rows certifies $\mathrm{OPT}=0$ if and only if there is a margin separator for $F=E_n$, namely an affine function of $\mathbf{z}$ plus a $\mu$-combination of the functions $\phi_h(1,\cdot)$, $h\in\mathcal C$, satisfying \eqref{eq:margin}, that is
$r(\mathbf{z})\leq -1$ for $\mathbf{z}\in E_n$ and $r(\mathbf{z})\geq 1$ for $\mathbf{z}\in O_n$. Fixing $\tau=1$ in a Generalized encoding function yields either a constant, which is absorbed into the affine part, or again a Generalized encoding function of $\mathbf{z}$. Hence $r$ has exactly the form treated in Lemma~\ref{lem:parity-rows}, with at most $|\mathcal C|$ Generalized terms, and the displayed conditions are exactly its hypothesis. Lemma~\ref{lem:parity-rows} then forces $|\mathcal C| \geq \tfrac12(4/3)^{\lceil n/4\rceil} =2^{\Omega(n)}$ and completes the proof.
\end{proof}

For the same instance, define the parity encoding $\phi^{\mathrm{par}}(\tau,\mathbf z) := \tau (-1)^{\sum_{p=1}^n z_p}$. Thus $\phi^{\mathrm{par}}(0,\mathbf 0)=0$, while $\phi^{\mathrm{par}}(1,\mathbf z)=1$ for $\mathbf z\in E_n$ and $\phi^{\mathrm{par}}(1,\mathbf z)=-1$ for $\mathbf z\in O_n$.


\begin{proof}[Proof of Proposition~\ref{prop:parity-upper-bound}]
Set $r(\mathbf z):=-\phi^{\mathrm{par}}(1,\mathbf z)$, so that $r(\mathbf z)=-1$ for $\mathbf z\in E_n$ and $r(\mathbf z)=1$ for $\mathbf z\in O_n$. This is a margin separator for $F=E_n$ in the sense of \eqref{eq:margin}, built from the single encoding function $\phi^{\mathrm{par}}$ (with no affine part). By Lemma~\ref{lem:gadget}, applied to the single-function family $\{\phi^{\mathrm{par}}\}$, the one row $\phi^{\mathrm{par}}$ certifies $\mathrm{OPT}=0$, that is, it closes the gap. By Theorem~\ref{thm:generalized-encoding-lower-bound}, every set of Generalized rows closing the gap of this instance has cardinality at least $\tfrac12(4/3)^{\lceil n/4\rceil}>1$ for $n\geq 12$, whereas the parity row closes it alone. Hence the parity row is not affinely equivalent to any single Generalized row, nor to any subexponential collection of them.
\end{proof}

\section{Test instances and computational setup}
\label{sec:test-instances}

\subsection{Problem type}
\paragraph{Stable-set instances.}
\label{subsec:test-stable-set}

For the stable-set family, each block $i\in V$ contains a local conflict graph $G_i=(U_i,F_i)$, where $U_i$ is the set of local binary variables and $F_i$ is the set of local conflict edges. A local feasible solution is a stable set of $G_i$: $ X^i_{\mathrm{SS}}:=\left\{\mathbf{x}^i\in\B^{U_i}:\ x_u^i+x_v^i\le 1 \quad \forall\,\{u,v\}\in F_i\right\}$.

Arranging the blocks above on a tree $T=(V,E)$, the full coupled stable-set instance is then:
\begin{equation*}
\begin{aligned}
    \max \quad
        & \sum_{i\in V} \sum_{v\in U_i} x^i_v\\
    \text{s.t.}\quad
        & \mathbf{x}^i\in X^i_{\mathrm{SS}},
            && i\in V,\\
        & x_p^{e,i}=x_p^{e,j},
            && e=ij\in E,\ p\in Q^e.
\end{aligned}
\end{equation*}
The local conflict constraints remain inside the blocks, while the equalities identify selected variables across adjacent blocks.

In~\cite{cdx-ipco} it was shown that, for packing IPs, requiring the projection of the feasible region onto the boundary variables to agree across every edge yields convergence of the monomial method, with explicit approximation guarantees for subfamilies of monomial encodings of bounded degree. Moreover, following Remark~\ref{rem:modeling}, we introduce a simple conflict-propagation step along the coupling edges before constructing the local optimization models. Consider an edge $e=ij\in E$ and two boundary variables $p,q\in Q^e$. If the corresponding variables are adjacent in the local conflict graph of block $i$, then the same conflict is inserted between the paired boundary variables of block $j$, and vice verse. This propagation is applied in both directions over all block-tree edges until no new local conflict edge is added. This conflict resolution achieves that the projections of the two blocks onto the shared boundary $Q^e$ agree. The preprocessing does not change the feasible solutions of the complete coupled instance, but it strengthens the individual block models.

\paragraph{Dominating-set instances.}\label{subsec:test-dominating-set} For the dominating-set family, each block $i\in V$ again contains a local graph $G_i=(U_i,F_i)$. For $v\in U_i$, let $N_i(v):=\{u\in U_i:\{u,v\}\in F_i\}$ be the open neighborhood of $v$ inside block $i$. A local feasible solution is a dominating set of $G_i$: $ X^i_{\mathrm{DS}}:= \left\{         \mathbf{x}^i\in\B^{U_i}: x^i_v+ \sum_{u\in N_i(v)} x_u^i\ge 1 \quad \forall v\in U_i \right\}$.

Arranging the blocks above on a tree $T = (V, E)$, the full coupled dominating-set instance is then:
\begin{equation*}
\begin{aligned}
    \min \quad
        & \sum_{i\in V} \sum_{v\in U_i} x^i_v \\
    \text{s.t.}\quad
        & \mathbf{x}^i\in X^i_{\mathrm{DS}},
            && i\in V,\\
        & x_p^{e,i}=x_p^{e,j},
            && e=ij\in E,\ p\in Q^e.
\end{aligned}
\end{equation*}
As in the stable-set case, the only constraints coupling different blocks are the equality constraints on the selected boundary variables.

For the dominating-set instances, we use the same principle as in Remark~\ref{rem:modeling}: implications induced by the coupling equalities are made explicit whenever they can be expressed as local constraints. Here, this is achieved by a simple neighborhood-dominance preprocessing step along each coupling edge. For an edge $e=ij\in E$ and a boundary variable $p\in Q^e$, let  $N_i(p)$ and $N_j(p)$ denote the open neighborhoods of the paired local variables in blocks $i$ and $j$. If $N_j(p)$ lies entirely in the boundary of $e$, it can be translated through the coupling map to block $i$. Whenever the translated neighborhood is a strict subset of $N_i(p)$, we replace the local neighborhood of $p$ in block $i$ by this smaller neighborhood. In terms of domination constraints, this replaces $x_p^{e,i}+\sum_{u\in N_i(p)} x_u^i \geq 1$ by the stronger local inequality $x_p^{e,i} +\sum_{u\in  N'_j(p)} x_u^i \geq 1$, where $N'_j(p)\subsetneq N_i(p)$ is the neighborhood of the paired variable in block $j$, translated to the local indexing of block $i$. The procedure is applied in both directions over all block-tree edges until no neighborhood can be reduced further. As in the stable-set case, the preprocessing preserves the feasible set of the full coupled formulation, while strengthening the local block descriptions used by the decomposition algorithm.

\subsection{Coupling equalities and topologies of $T=(V,E)$}
\label{subsec:test-couplings}

A coupling-density parameter specifies how many equality constraints are placed on each edge of the block tree. We denote this number by $q$. For each edge $e=ij\in E$, we select two ordered lists of local indices, one in block $i$, $I^{e,i} =(\iota^{e,i}_1,\ldots,\iota^{e,i}_q)$, and one in block $j$, $ I^{e,j}=(\iota^{e,j}_1,\ldots,\iota^{e,j}_q)$, and impose the pairwise equalities $x^i_{\iota^{e,i}_p} = x^j_{\iota^{e,j}_p}$, $p=1,\ldots,q$. Equivalently, $Q^e=\{1,\ldots,q\}$ indexes the coupling positions on edge $e$, and $x_p^{e,i}$ denotes the local variable $x^i_{\iota^{e,i}_p}$. The choice of the ordered index lists depends on the block-tree topology. Fix $n:=|U_i|$ as the number of local variables in all blocks $i$.

\paragraph{Star.}
The star topology has center block $1$ and edges $e=1j$, $j=2,\ldots,|V|$. For every such edge, the same first $q$ variables of the center block are coupled with the first $q$ variables of the leaf block: $I^{1j,1}=(1,\ldots,q)$, $I^{1j,j}=(1,\ldots,q)$, for  $j=2,\ldots,|V|$. Thus the same center variables are reused across all incident edges. This creates the usual two-stage structure in which all leaf blocks share the same first-stage variables.

\paragraph{Path.}
For the path topology, the edges are $e=i(i+1)$, $i=1,\ldots,|V|-1$. The last $q$ variables of block $i$ are coupled with the first $q$ variables of block $i+1$: $I^{i(i+1),i}=(n-q+1,\ldots,n)$, $I^{i(i+1),i+1} = (1,\ldots,q)$. Hence, for each $p=1,\ldots,q$, the equality is $x^i_{n-q+p}=x^{i+1}_{p}$. This convention makes the right boundary of one block coincide with the left boundary of the next block.

\paragraph{Binary tree.}
For the binary-tree topology, with the natural ordering of nodes, each parent block $i$ is connected to its children $2i$ and $2i+1$, whenever these indices are present. For each parent--child edge $e=ij$, the last $q$ variables of the parent are coupled with the first $q$ variables of the child: $ I^{ij,i} = (n-q+1,\ldots,n)$, $I^{ij,j} = (1,\ldots,q)$. Thus, if a block has two children, the same last $q$ parent variables are coupled to both children. This produces a branching analogue of the path construction found in multi-stage stochastic programming problems.

\paragraph{Random tree.}
For the random-tree topology, the block tree is sampled randomly on a prescribed number of blocks. For each sampled edge $e=ij$, we choose $q$ distinct local indices from block $i$ and $q$ distinct local indices from block $j$, uniformly at random and independently across the two endpoints, i.e., $I^{e,i}\subseteq U_i$, $I^{e,j}\subseteq U_j$, $|I^{e,i}|=|I^{e,j}|=q$. The order of the sampled lists determines the pairing of the equality constraints. The sampling is performed separately for each edge, so a local variable of a high-degree block may be selected in more than one incident coupling edge.

\subsection{Instance generation}
\label{subsec:test-generation}

All computational experiments use instances with $|V|=15$ blocks. Each block contains $100$ local binary variables. For both the stable-set and dominating-set families, the local graph in each block is generated by sampling $500$ edges uniformly without replacement from the complete graph on $100$ nodes. Thus, each local graph is a sparse random graph with the same number of vertices and edges. The random seed of each block is derived from the global instance seed and the block index, so that different blocks are generated independently but reproducibly.

For each problem family, we consider three coupling sizes $q\in\{20,30,40\}$. For each value of $q$, we generate instances on the four deterministic or random block-tree topologies described above: star, path, binary tree, and random tree. In addition, we generate stochastic variants for the star and binary-tree topologies.

The non-stochastic instances use the same objective scaling in all blocks. In the stochastic variants, the objective coefficients are scaled to mimic repeated or inherited decisions across stages. For the star topology, the central block is assigned weight $|V|-1$, while every leaf block has weight $1$. For the binary-tree topology, the weight of a block depends on its depth. With $|V|=15$, the tree has four levels; a block at level $t$ receives weight $2^{4-t}$, so the root has weight $8$, its children have weight $4$, the next level has weight $2$, and the leaves have weight $1$.

We use five independent seeds for each combination of problem family, topology, coupling size, and stochastic flag. The stochastic flag is only used for the star and binary-tree topologies. Thus, for each problem family and each coupling size $q$, the test set contains four non-stochastic topology classes and two stochastic topology classes.

\subsection{Computational setup}
\label{subsec:computational-setup}

All algorithms were implemented in Python 3.13 using Gurobi 13.0.1 as the MIP solver. The experiments were run on Amazon EC2 \texttt{c8a.4xlarge} instances with 16 vCPUs and 32 GB of RAM. Each run was assigned a time limit of $1800$ seconds. The monolithic model had an additional 26 GB soft memory limit to avoid crashing. Instances that reached either of these limits were tagged with status \texttt{TimeLimit} and \texttt{MemoryLimit}, respectively.

The experiments solve one benchmark instance at a time. For the CRG algorithms, the pricing phase exploits the block structure by solving the 15 block pricing subproblems in parallel, using one Gurobi thread per pricing model. The monolithic formulation is solved with 16 Gurobi threads. The auxiliary separation MIPs used by the Monomial, Reflected, and Generalized strategies are also solved with 16 Gurobi threads.

Each CRG run is initialized from a short monolithic solve. Specifically, the monolithic formulation is solved with a deterministic Gurobi work limit of $10$, which suffices to find a feasible solution. Its restriction to each block is used to create one initial column per block, and its objective value initializes the primal bound. We also solve the LP relaxation of the monolithic model and use its objective value as the initial dual bound.

The restricted master problem is solved as a linear program during column generation. A new column is added when its reduced cost is smaller than $-\varepsilon_{\mathrm{price}}=-10^{-4}$. After column generation terminates for the current set of generated encoding rows, the current master solution is aggregated by boundary signatures and the selected separation strategy is applied on each block-tree edge. The default row-separation tolerance is $\varepsilon_{\mathrm{cut}}=10^{-6}$ for the Vertex, Monomial, Reflected, and Generalized strategies.

The number of generated rows per separation call is controlled by a factor parameter. For the Vertex strategy, we use factor $1.0$, so at most $q$ violated signature equalities are added per edge and separation call, where $q$ is the number of coupling equalities per edge. For the Monomial, Reflected, and Generalized strategies, we use factor $0.5$; the corresponding auxiliary separation MIP stores up to $0.5q$ improving solutions from the Gurobi solution pool in each orientation of the edge.

We also consider a hybrid CRG strategy that combines the Monomial and Vertex separation mechanisms. The algorithm starts with the Monomial strategy, which searches for violated monomial equalities through the auxiliary separation problem. This phase is used for the first five row-generation rounds, adding up to $0.5q$ M-type rows per edge and orientation in each separation call. After this initial phase, the algorithm switches permanently to the Vertex strategy, which separates violated signature equalities and adds up to $q$ V-type rows per edge and separation call. Rows generated during the Monomial phase remain active after the switch and continue to contribute dual terms to the pricing problems.

The CRG algorithm stops with status \texttt{Optimal} when no columns and no rows are added. It also stops early with status \texttt{Gap\_Closed} when the relative gap between the current primal and dual bounds falls below $10^{-6}$. Because the tested objectives are integer-valued, we also declare convergence with status \texttt{Integer\_Gap\_Closed} when the primal and dual bounds have the same integer ceiling up to a tolerance of $10^{-6}$. After each row-generation round, a five-second MIP heuristic is run on the current master with binary column variables. At the end of the CRG run, the final generated master is also solved with binary column variables to update the best primal bound.

For the star topology, we also implemented two benchmark decomposition methods: the integer L-shaped method \cite{laporte-louveaux, angulo-ahmed-dey} and scenario decomposition~\cite{ahmed-scenario}. The integer L-shaped master problem is solved with 16 Gurobi threads, while its 14 scenario subproblems are solved in parallel using one Gurobi thread each. Scenario decomposition solves 15 block subproblems in parallel, again using one Gurobi thread per subproblem. Thus, both implementations use the available threads in a manner comparable to the CRG algorithms, while respecting the different subproblem structure of each method.

\section{Computational results}
\label{sec:computational-results}

This section evaluates the computational behavior of the proposed CRG variants on the stable-set and dominating-set test beds. The benchmark contains 180 instances, 90 for each problem. Each instance is specified by a decomposition topology, a coupling level, a stochastic flag when applicable, and a random seed. The common benchmark compares the five CRG variants against the monolithic formulation on all instances. The integer L-shaped and scenario-decomposition baselines are reported separately because they are available only for the star topology.

We use the following status convention throughout the section. Runs with status \texttt{Optimal}, \texttt{Integer\_Gap\_Closed}, or \texttt{Gap\_Closed} are treated as solved, and their final gap is set to zero. Runs with status \texttt{TimeLimit} or \texttt{MemoryLimit} are treated as unsolved and assigned a censored time limit of 1800 seconds. The reported average gap, denoted by \emph{Gap}, therefore averages the final gaps after setting solved runs to zero. The column \emph{Unsolved gap} reports the average gap restricted to runs that did not close. All gaps are reported in percentage points.

The primary ordering criterion is lexicographic: number of solved runs, average censored time, and final gap. This criterion is used to order the solvers within each problem in Table~\ref{tab:common-benchmark}. Under this ordering, a method that closes more runs is preferred even if another method produces a smaller average residual gap on the runs that remain unsolved.

\begin{table}[ht]
\small
\centering
\caption{Common benchmark over all topologies.}
\label{tab:common-benchmark}
\begin{tabular}{llrrrrr}
\toprule
Problem & Solver & Solved & Time & Time solved & Gap & Unsolved gap \\
 & &  & (s) & (s) & (\%) & (\%) \\
\midrule
Dominating set & Vertex & 19/90 & 1607 & 888 & 6.22 & 7.88 \\
Dominating set & Hybrid & 14/90 & 1682 & 1044 & 5.96 & 7.06 \\
Dominating set & Monomial & 13/90 & 1683 & 990 & 6.04 & 7.06 \\
Dominating set & Generalized & 4/90 & 1758 & 849 & 7.84 & 8.21 \\
Dominating set & Reflected & 4/90 & 1759 & 886 & 7.95 & 8.32 \\
Dominating set & Monolithic & 0/90 & 1800 & -- & 13.52 & 13.52 \\
\addlinespace
Stable set & Hybrid & 50/90 & 963 & 293 & 2.31 & 5.20 \\
Stable set & Monomial & 49/90 & 975 & 285 & 2.30 & 5.06 \\
Stable set & Vertex & 47/90 & 993 & 255 & 2.59 & 5.41 \\
Stable set & Generalized & 32/90 & 1201 & 115 & 3.51 & 5.44 \\
Stable set & Reflected & 31/90 & 1205 & 72 & 3.71 & 5.66 \\
Stable set & Monolithic & 14/90 & 1535 & 100 & 8.11 & 9.60 \\
\bottomrule
\end{tabular}
\end{table}

Table~\ref{tab:common-benchmark} shows a clear separation between the two problem classes. The stable-set instances are substantially easier: the best methods close slightly more than half of the runs, whereas the best dominating-set method closes only 19 out of 90 runs. For the stable-set benchmark, Hybrid has the best lexicographic performance, followed closely by Monomial and Vertex. Monomial obtains the smallest average gap, but Hybrid closes one additional run and has the smallest average censored time. For the dominating-set benchmark, Vertex is the best method under the lexicographic criterion because it closes the largest number of runs and has the smallest average censored time. However, Hybrid and Monomial provide stronger residual bounds on the unsolved runs. Thus, Vertex is the most effective method when the objective is to close runs, while Hybrid and Monomial are preferable when the quality of the final bound is the main criterion.

The monolithic formulation is not competitive on this benchmark. It closes only 14 stable-set runs and no dominating-set runs, and it has the largest average gaps in both problem classes. Its poor behavior is especially pronounced for the dominating-set instances, where all monolithic runs reach the time or memory limit.

\begin{figure}[tbp]
\centering
\begin{subfigure}{0.7\textwidth}
  \centering
  \includegraphics[width=\linewidth]{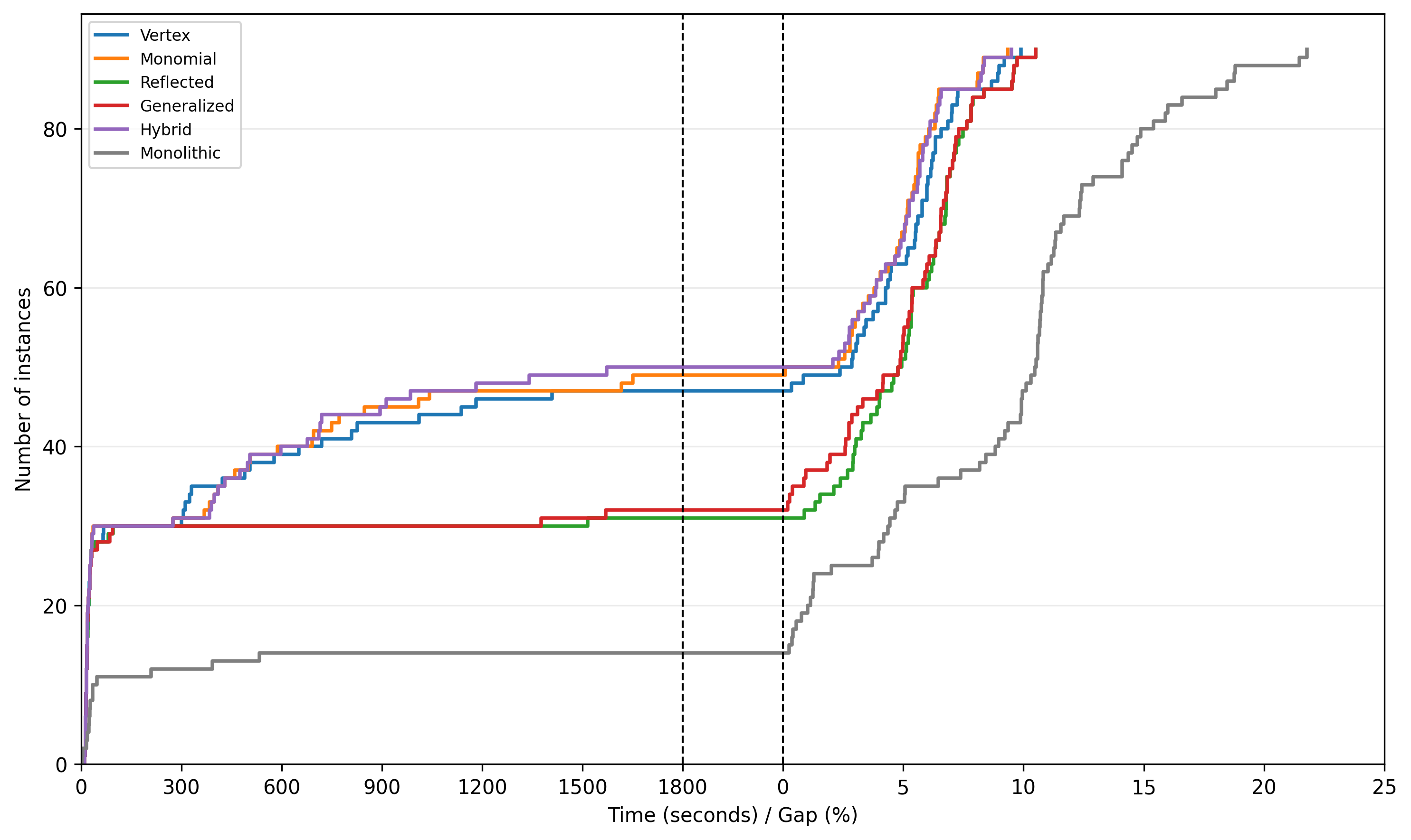}
  \caption{Stable set}
  \label{fig:profiles-common-stable}
\end{subfigure}

\vspace{0.5em}

\begin{subfigure}{0.7\textwidth}
  \centering
  \includegraphics[width=\linewidth]{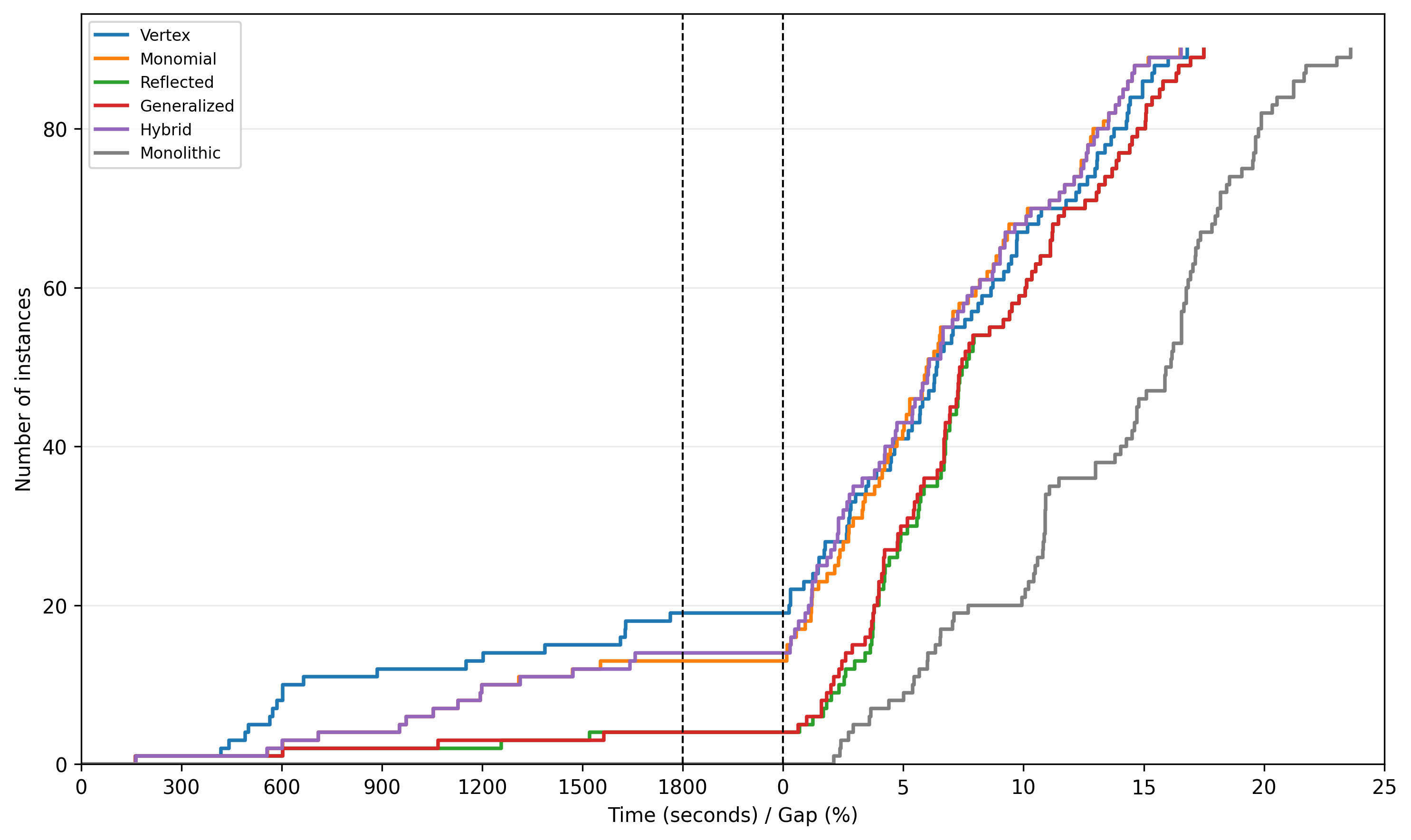}
  \caption{Dominating set}
  \label{fig:profiles-common-dominating}
\end{subfigure}
\caption{Common performance profiles. The left side of each panel orders
solved instances by running time; the right side orders unsolved instances by
final gap.}
\label{fig:profiles-common}
\end{figure}

Figure~\ref{fig:profiles-common} gives the distributional view behind the aggregate statistics. The left side of each profile orders solved instances by running time, whereas the right side orders unsolved instances by final gap. In the stable-set profile (Figure~\ref{fig:profiles-common-stable}), the CRG variants dominate the monolithic baseline by closing many more instances and by producing smaller residual gaps when they do not close. The relative ordering among Hybrid, Monomial, and Vertex is close, which is consistent with Table~\ref{tab:common-benchmark}. In the dominating-set profile (Figure~\ref{fig:profiles-common-dominating}), the separation between Vertex and the other CRG variants is mainly on the solved side of the plot, while Hybrid and Monomial remain competitive on the gap side. This confirms that no single CRG variant dominates the others uniformly: the best method depends on whether the priority is closing instances or obtaining the strongest residual bound.

\subsection{Iteration counts and convergence}
\label{subsec:iterations-convergence}

Table~\ref{tab:crg-iterations} reports average iteration counts for the five CRG variants on the common benchmark. Outer iterations correspond to rounds in which the master is strengthened by newly generated cuts. Inner iterations correspond to pricing iterations and therefore capture most of the repeated subproblem work. The ratio \emph{Inner/outer} gives a coarse measure of how much pricing work is performed per outer round.

\begin{table}[ht]
\small
\centering
\caption{Average CRG iteration counts.}
\label{tab:crg-iterations}
\begin{tabular}{llrrrrr}
\toprule
Problem & Solver & Solved & Outer it. & Inner it. & Inner/outer & Cuts \\
\midrule
Dominating set & Vertex & 19/90 & 3.0 & 365.9 & 129.0 & 788 \\
Dominating set & Hybrid & 14/90 & 6.3 & 490.8 & 79.6 & 516 \\
Dominating set & Monomial & 13/90 & 6.9 & 506.2 & 76.3 & 325 \\
Dominating set & Generalized & 4/90 & 2.1 & 405.8 & 198.2 & 400 \\
Dominating set & Reflected & 4/90 & 2.0 & 396.3 & 196.4 & 390 \\
\addlinespace
Stable set & Hybrid & 50/90 & 6.2 & 408.3 & 60.7 & 976 \\
Stable set & Monomial & 49/90 & 7.7 & 444.8 & 55.1 & 431 \\
Stable set & Vertex & 47/90 & 4.4 & 345.0 & 73.5 & 1338 \\
Stable set & Generalized & 32/90 & 2.7 & 337.5 & 118.2 & 566 \\
Stable set & Reflected & 31/90 & 2.5 & 333.4 & 123.9 & 509 \\
\bottomrule
\end{tabular}
\end{table}

The iteration statistics reveal two complementary patterns. First, Vertex usually uses fewer outer iterations than Monomial and Hybrid. On the dominating-set benchmark, Vertex performs only 3.0 outer iterations on average, compared with 6.3 for Hybrid and 6.9 for Monomial. Second, Monomial and Hybrid spread the work over more outer rounds and have lower inner-per-outer ratios. This behavior is consistent with their stronger residual gaps on unsolved dominating-set instances. Generalized and Reflected complete relatively few outer rounds that require many inner iterations per outer round and do not obtain comparable closure rates or gaps.

Figure~\ref{fig:convergence-pair} illustrates these iteration patterns on two binary-tree instances. The top panel shows a stable-set instance with stochastic coupling, and the bottom panel shows a deterministic dominating-set instance with the same topology and coupling level. In each panel, solid curves report the dual bound, dashed step curves report the incumbent primal bound, markers indicate the end of outer iterations, and the lower subpanel records the cuts added at those rounds. The stable-set instance (Figure~\ref{fig:convergence-pair-stable}; maximization form) is a nontrivial case in which Vertex closes first, Monomial and Hybrid close later, and Generalized and Reflected reach the time limit. The dominating-set instance (Figure~\ref{fig:convergence-pair-dominating}; minimization form) is harder: no CRG method closes, so the comparison is driven by the final residual gap. The trajectories show how the dual bound evolves between cut-injection rounds and how primal improvements occur only intermittently. In both cases, the methods that generate the most cuts are not necessarily the methods with the best final outcome, which motivates a separate look at where and how many cuts are generated.

\begin{figure}[tbp]
\centering
\begin{subfigure}{0.7\textwidth}
  \centering
  \includegraphics[width=\linewidth]{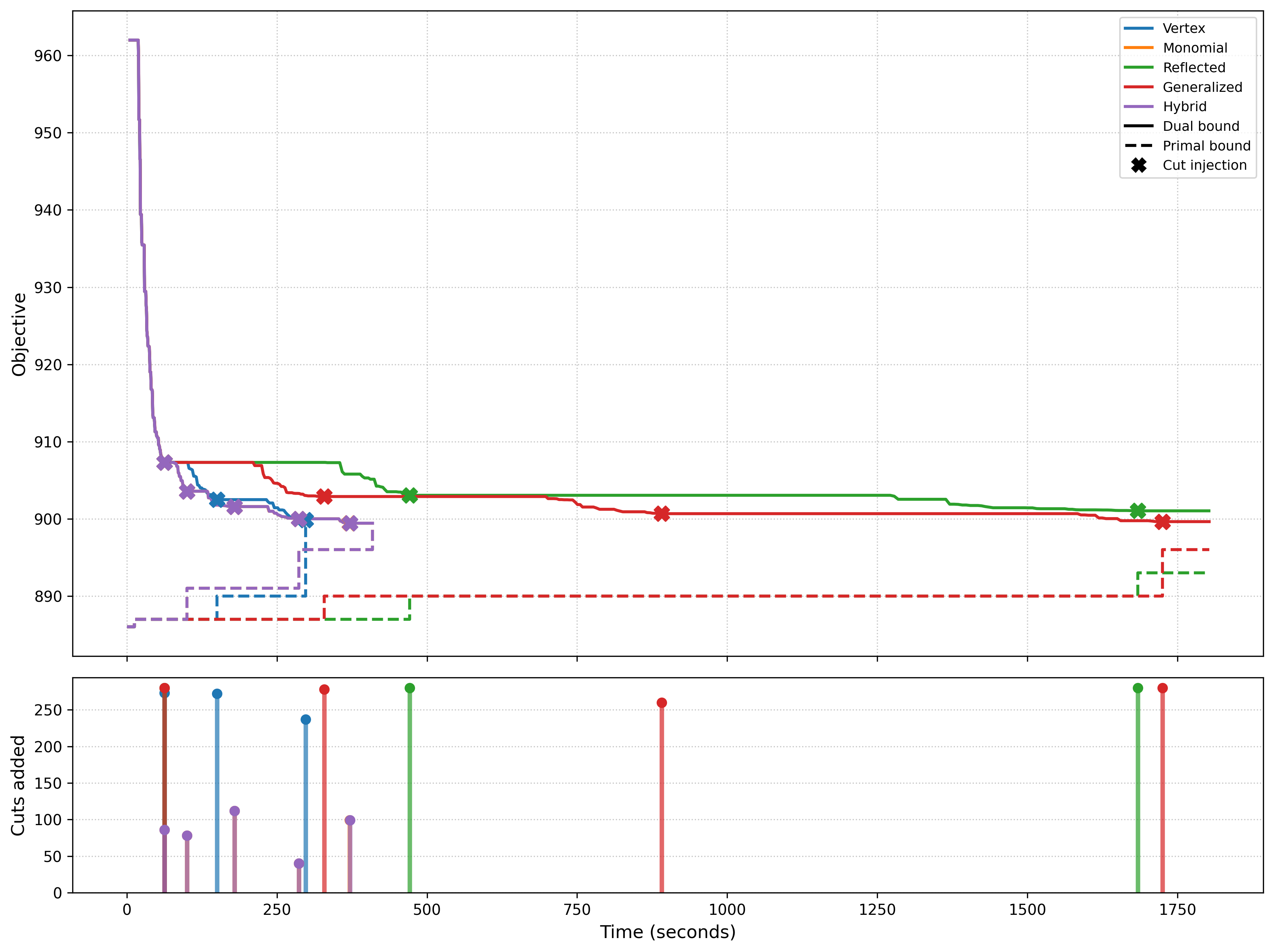}
  \caption{Stable set}
  \label{fig:convergence-pair-stable}
\end{subfigure}

\vspace{0.5em}

\begin{subfigure}{0.7\textwidth}
  \centering
  \includegraphics[width=\linewidth]{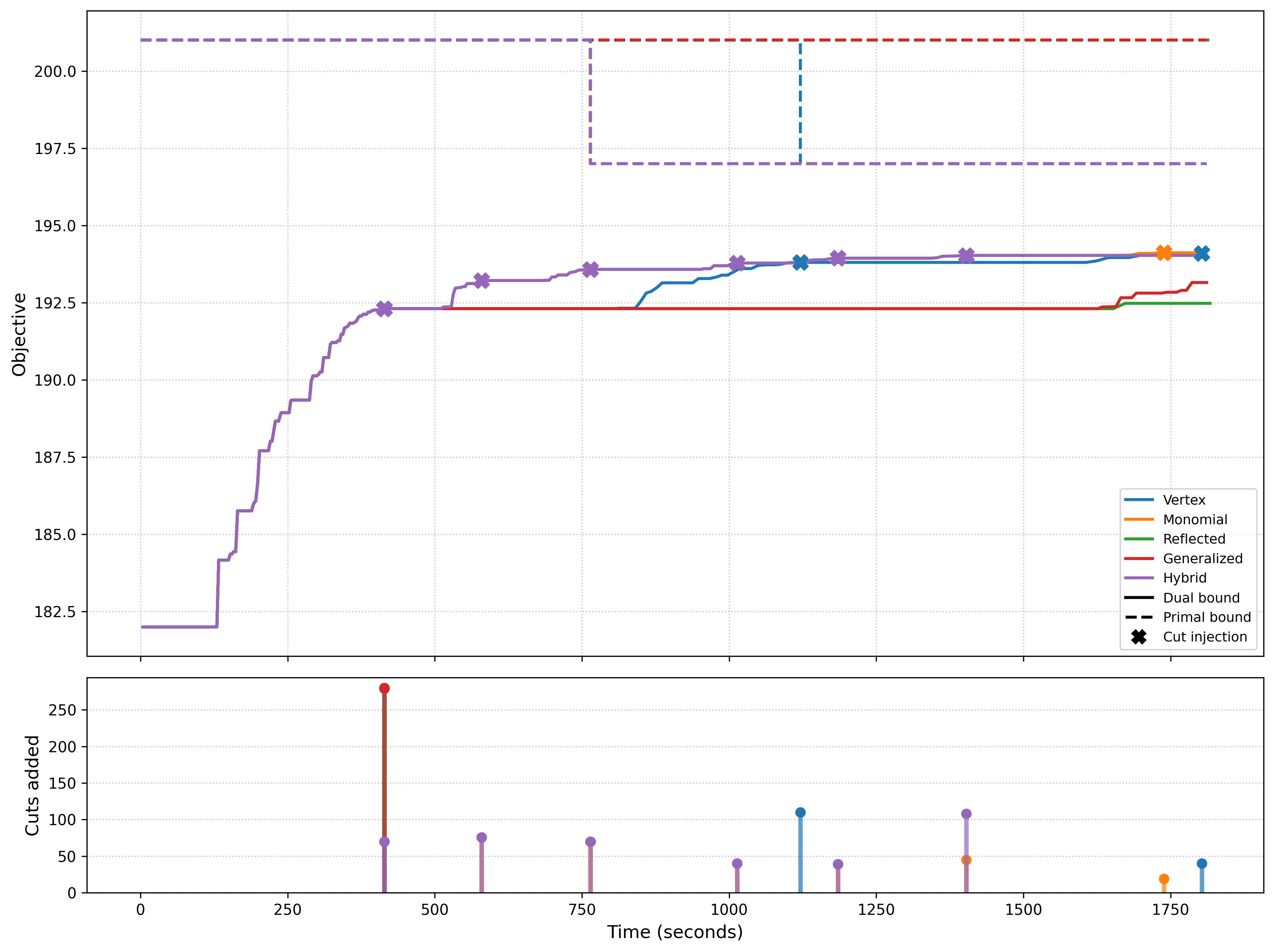}
  \caption{Dominating set}
  \label{fig:convergence-pair-dominating}
\end{subfigure}
\caption{Representative convergence trajectories.}
\label{fig:convergence-pair}
\end{figure}

\subsection{Effect of topology and coupling}
\label{subsec:topology-coupling}

The aggregate comparison hides a strong dependence on the structure of the instance. Table~\ref{tab:instance-effects} reports the effect of topology and coupling on the common benchmark. The star topology is the easiest topology for both problems, but the effect is much stronger for the stable-set instances. On the stable-set benchmark, the common solvers (the five CRG variants and the monolithic formulation) close 164 out of 180 star runs, whereas the non-star topologies are much harder. On the dominating-set benchmark, even the star topology remains difficult: only 35 out of 180 common solver runs close.

Coupling is the most visible difficulty driver. For the dominating-set problem, increasing the coupling level from 20 to 30 reduces the closure rate from 47/180 to 6/180, and coupling level 40 leaves only one closed run among 180. The stable-set instances are less sensitive in terms of closure counts because the star instances remain easy, but the average gap increases monotonically with the coupling level.

The stochastic flag has a milder effect than topology or coupling. Restricting the comparison to the topologies where both deterministic and stochastic variants are present, stochasticity is nearly neutral for the stable-set problem. For the dominating-set problem, stochastic instances are easier on average, mostly because the deterministic binary-tree cases are particularly hard.

\begin{table}[ht]
\small
\centering
\caption{Effect of topology, coupling, and the stochastic flag on the common benchmark.}
\label{tab:instance-effects}
\begin{tabular}{lllrrrr}
\toprule
Factor & Problem & Value & Solved & Time & Gap & Unsolved gap \\
 & & & & (s) & (\%) & (\%) \\
\midrule
Topology & Dominating set & star & 35/180 & 1598 & 3.87 & 4.80 \\
Topology & Dominating set & path & 10/90 & 1752 & 10.59 & 11.92 \\
Topology & Dominating set & binary tree & 6/180 & 1783 & 9.36 & 9.69 \\
Topology & Dominating set & random tree & 3/90 & 1776 & 10.47 & 10.83 \\
\addlinespace
Topology & Stable set & star & 164/180 & 186 & 0.17 & 1.91 \\
Topology & Stable set & binary tree & 37/180 & 1578 & 5.33 & 6.71 \\
Topology & Stable set & path & 15/90 & 1626 & 5.69 & 6.82 \\
Topology & Stable set & random tree & 7/90 & 1719 & 5.84 & 6.33 \\
\midrule
Coupling & Dominating set & 20 & 47/180 & 1569 & 4.42 & 5.98 \\
Coupling & Dominating set & 30 & 6/180 & 1777 & 7.72 & 7.98 \\
Coupling & Dominating set & 40 & 1/180 & 1799 & 11.63 & 11.69 \\
\addlinespace
Coupling & Stable set & 20 & 105/180 & 961 & 1.88 & 4.50 \\
Coupling & Stable set & 30 & 58/180 & 1265 & 3.93 & 5.80 \\
Coupling & Stable set & 40 & 60/180 & 1210 & 5.46 & 8.19 \\
\midrule
Stochastic & Dominating set & true & 33/180 & 1624 & 5.87 & 7.18 \\
Stochastic & Dominating set & false & 8/180 & 1757 & 7.36 & 7.71 \\
\addlinespace
Stochastic & Stable set & true & 103/180 & 854 & 2.78 & 6.49 \\
Stochastic & Stable set & false & 98/180 & 910 & 2.72 & 5.98 \\
\bottomrule
\end{tabular}
\end{table}

\subsection{Star topology and specialized baselines}
\label{subsec:star-specialized}

The integer L-shaped and scenario-decomposition baselines are meaningful only on the star topology. Therefore, they are excluded from the common performance profiles and reported separately in Table~\ref{tab:star-baselines}. On stable-set star instances, all CRG variants close all runs, but the two specialized baselines are substantially faster. Integer L-shaped closes all 30 runs in 4.6 seconds on average, and scenario decomposition closes all 30 runs in 7.0 seconds on average. The fastest CRG methods require about 20 seconds on average.

The same qualitative conclusion does not fully carry over to the dominating-set star instances. Integer L-shaped remains the strongest method, closing 20 out of 30 runs. Scenario decomposition closes fewer runs than Integer L-shaped, although its average gap is slightly smaller. Among the CRG methods, Vertex is the strongest by the lexicographic criterion, while Hybrid and Monomial are competitive in terms of final gap.

Overall, these comparisons are encouraging for CRG: although it does not exploit the specialized master--recourse structure available on star instances, its performance remains competitive with methods designed specifically for that setting, particularly on the harder dominating-set instances.

\begin{table}[ht]
\small
\centering
\caption{Star-topology comparison with specialized baselines.}
\label{tab:star-baselines}
\begin{tabular}{llrrrr}
\toprule
Problem & Solver & Solved & Time & Gap & Unsolved gap \\
 & & & (s) & (\%) & (\%) \\
\midrule
Dominating set & Integer L-shaped & 20/30 & 740 & 2.08 & 6.24 \\
Dominating set & Scenario decomp. & 12/30 & 1298 & 1.97 & 3.28 \\
Dominating set & Vertex & 11/30 & 1327 & 2.50 & 3.95 \\
Dominating set & Hybrid & 9/30 & 1550 & 2.46 & 3.52 \\
Dominating set & Monomial & 7/30 & 1560 & 2.63 & 3.44 \\
Dominating set & Generalized & 4/30 & 1673 & 4.01 & 4.62 \\
Dominating set & Reflected & 4/30 & 1678 & 4.02 & 4.64 \\
Dominating set & Monolithic & 0/30 & 1800 & 7.58 & 7.58 \\
\addlinespace
Stable set & Integer L-shaped & 30/30 & 5 & 0.00 & -- \\
Stable set & Scenario decomp. & 30/30 & 7 & 0.00 & -- \\
Stable set & Monomial & 30/30 & 20 & 0.00 & -- \\
Stable set & Hybrid & 30/30 & 20 & 0.00 & -- \\
Stable set & Vertex & 30/30 & 22 & 0.00 & -- \\
Stable set & Reflected & 30/30 & 24 & 0.00 & -- \\
Stable set & Generalized & 30/30 & 24 & 0.00 & -- \\
Stable set & Monolithic & 14/30 & 1006 & 1.02 & 1.91 \\
\bottomrule
\end{tabular}
\end{table}

\subsection{Cut-generation behavior}
\label{subsec:cut-behavior}

The preceding tables and performance profiles compare the final outcomes of the algorithms. Figure~\ref{fig:heatmaps} instead illustrates how the CRG variants generate the rows used to enforce consistency between adjacent blocks. Each row of a heatmap corresponds to an outer iteration and each column to a coupling edge. The top panels report the average support size of the signatures generated on each edge, whereas the bottom panels report the corresponding number of cuts.

Figure~\ref{fig:heatmap-dominating} shows a representative hard dominating-set instance on a binary tree with coupling level 20. None of the CRG variants closes this instance, so their final residual gaps provide the relevant secondary comparison. Generalized and Reflected generate many relatively dense cuts in their first outer iterations, but this activity does not translate into a comparably strong final bound. Monomial and Hybrid distribute their cut generation over more outer iterations and finish with smaller residual gaps. Vertex is more aggressive and closes more dominating-set instances overall, although this example also illustrates that the number of closed instances and the quality of the residual bound need not induce the same ranking.

Figure~\ref{fig:heatmap-stable} presents a stable-set instance for which the algorithms exhibit different closing behavior. Vertex closes first, followed by Monomial and Hybrid, whereas Generalized and Reflected reach the time limit. The latter two methods again generate a large number of dense signatures early in the solution process. Thus, the heatmaps suggest that generating more cuts, or cuts with larger supports, does not by itself produce faster convergence. The structural diversity of the generated signatures also appears to be relevant.

\begin{figure}[tbp]
\centering
\begin{subfigure}{\textwidth}
  \centering
  \includegraphics[width=\linewidth]{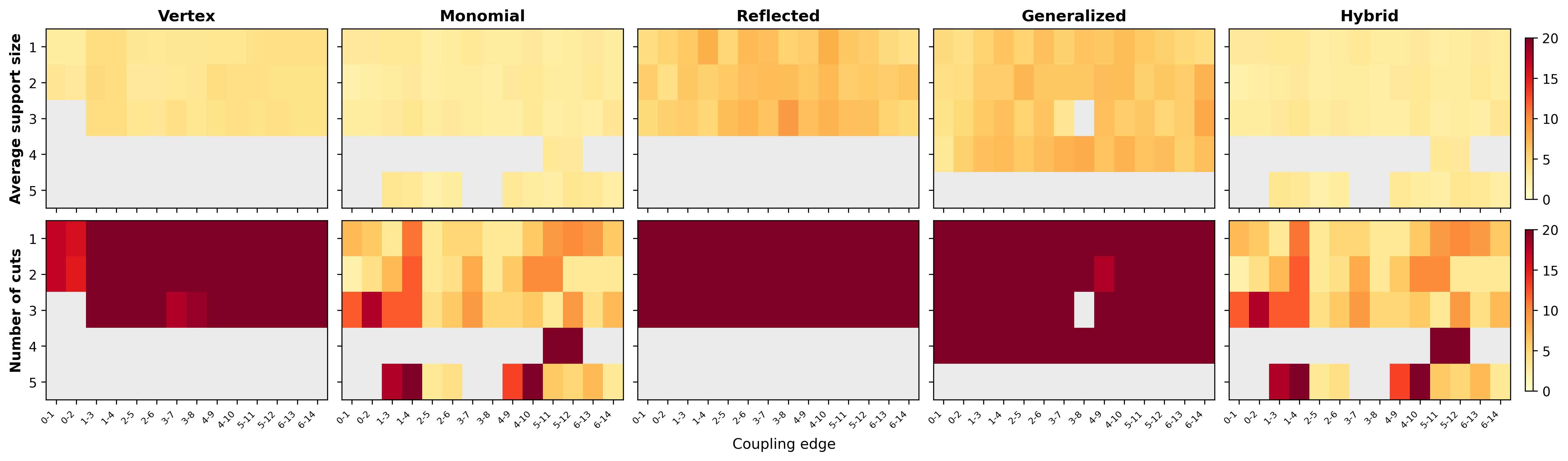}
  \caption{Stable set, binary tree, $q=20$, stochastic}
  \label{fig:heatmap-stable}
\end{subfigure}

\vspace{0.5em}

\begin{subfigure}{\textwidth}
  \centering
  \includegraphics[width=\linewidth]{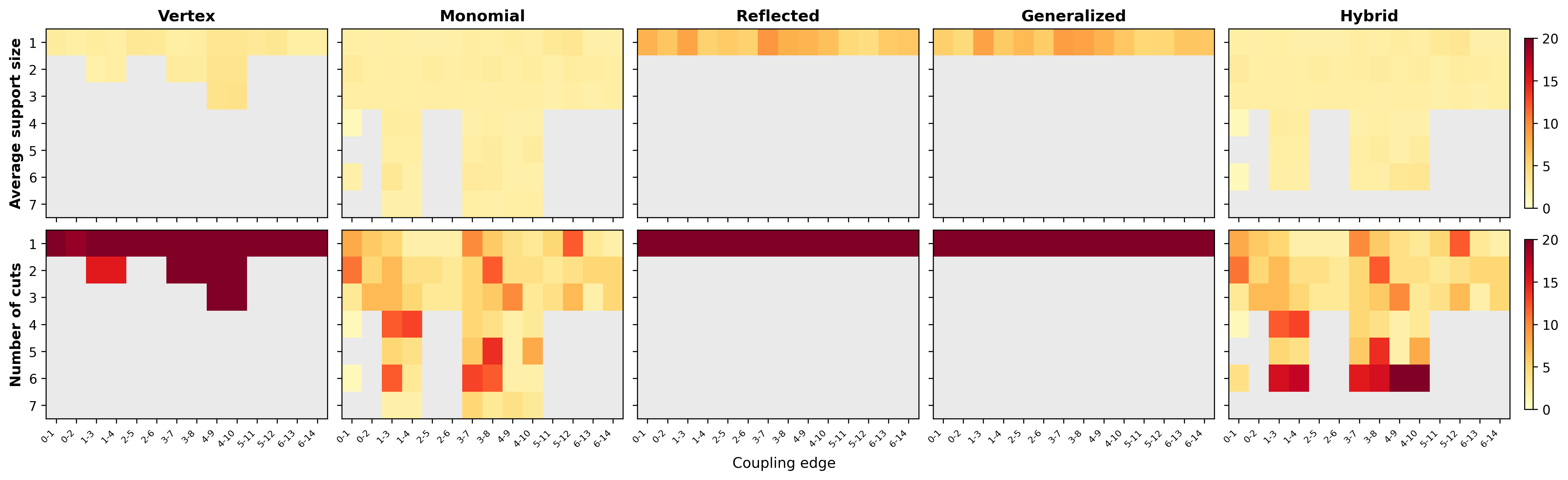}
  \caption{Dominating set, binary tree, $q=20$, deterministic}
  \label{fig:heatmap-dominating}
\end{subfigure}
\caption{Cuts and average signature support size per outer iteration and coupling edge.}
\label{fig:heatmaps}
\end{figure}

To examine this issue systematically, we use the signatures recorded in the cut-generation logs. Recall from Section~\ref{subsec:test-couplings} that every coupling edge $e$ carries $q$ pairs of boundary variables, indexed by $Q^{e}=\{1,\ldots,q\}$, and from Section~\ref{sec:strategies} that a generated Monomial or Reflected row is specified by a nonempty support $S\subseteq Q^{e}$, a Generalized row by a pair of disjoint sets $S^{+},S^{-}\subseteq Q^{e}$, and a Vertex row by a full assignment $\mathbf{s}\in\B^{Q^{e}}$. We call this defining data, together with the edge, the \emph{signature} of the row, and we refer to the elements of $S^{+}$ and $S^{-}$ as the positive and negative \emph{literals} of a Generalized row. The diversity analysis below concerns the three set-based encodings.

The \emph{literal set} of a generated row $h$ is $\mathcal L(h):=S$ for a Monomial or Reflected row with support $S$, and $\mathcal L(h):=\{(p,+):p\in S^{+}\}\cup\{(p,-):p\in S^{-}\}$ for a Generalized row, so that Monomial and Reflected literal sets are unsigned while Generalized literal sets retain signs. The \emph{degree} of a row is $|\mathcal L(h)|$ and its \emph{density} is $|\mathcal L(h)|/q$. For two rows $h$ and $h'$ generated on the same edge, we measure structural dissimilarity through the Jaccard index~\cite{tan-dm}
$$
J(h,h'):=1-\frac{|\mathcal L(h)\cap\mathcal L(h')|}{|\mathcal L(h)\cup\mathcal L(h')|}\in[0,1].
$$
A pair is classified as \emph{nearly duplicate} when $J(h,h')\leq 0.2$, or equivalently when at least $80\%$ of their combined literals coincide. A pair is \emph{nested} when one of its literal sets is contained in the other.

Table~\ref{tab:signature-diversity-first} summarizes these statistics for the first outer iteration in which each encoding generates cuts. Recall that the separation problem is solved in both orientations of each coupling edge and may produce up to $0.5q$ rows per orientation, giving a total per-edge row budget of $q$. For each active instance--edge pair, ``Budget share'' is the number of generated signatures divided by $q$, while ``Density'' is the average signature support size divided by $q$. The remaining statistics are averages over signatures generated on the edge. These edge-level quantities are then averaged over active edges within each instance and finally across instances. This hierarchical aggregation gives equal weight to each instance, while normalizing by $q$ makes the budget and density columns comparable across coupling levels. An edge contributes only if it generates at least one signature.

\begin{table}[t]
\centering
\small
\setlength{\tabcolsep}{4pt}
\caption{Signature diversity and effective row-budget utilization in the first cut-producing outer iteration.}
\label{tab:signature-diversity-first}
\begin{tabular}{llrrrrr}
\toprule
Problem & Encoding & Budget share & Density & Jaccard & Near dup. & Nested \\
\midrule
Stable set & Monomial    & 0.196 & 0.094 & 0.705 & 0.015 & 0.193 \\
Stable set & Reflected   & 0.873 & 0.257 & 0.544 & 0.178 & 0.230 \\
Stable set & Generalized & 0.887 & 0.249 & 0.571 & 0.163 & 0.210 \\
Dominating set & Monomial    & 0.140 & 0.077 & 0.738 & 0.002 & 0.127 \\
Dominating set & Reflected   & 0.907 & 0.316 & 0.531 & 0.229 & 0.211 \\
Dominating set & Generalized & 0.911 & 0.311 & 0.537 & 0.219 & 0.203 \\
\bottomrule
\end{tabular}
\end{table}

The differences are substantial. On an active edge in the first iteration, Monomial uses on average only 19.6\% of the available row budget for stable set and 14.0\% for dominating set. Reflected and Generalized use between 87.3\% and 91.1\%, indicating that their separation procedures typically return close to the maximum permitted number of rows. Their signatures are also much denser: the average Monomial density is below 0.10, whereas the average density of Reflected and Generalized ranges from 0.25 to 0.32.

Despite using a much smaller fraction of the row budget and producing sparser signatures, Monomial exhibits a substantially larger average pairwise Jaccard distance, 0.71--0.74 compared with 0.53--0.57. Moreover, its average fraction of nearly duplicate pairs is below 2\% for both problem classes, whereas between 16\% and 23\% of the Reflected and Generalized pairs are nearly duplicate. Thus, the additional rows produced by Reflected and Generalized are not proportionally more diverse; their high budget utilization is associated with denser and more mutually similar signatures.

The nesting statistic provides more qualified evidence. For dominating set, nested pairs are less frequent under Monomial, with a fraction of 0.13 compared with 0.20--0.21. For stable set, the differences are more modest, with fractions between 0.19 and 0.23. The most consistent distinction is therefore the combination of higher budget utilization, higher density, smaller pairwise Jaccard distance, and a larger fraction of nearly duplicate pairs under Reflected and Generalized.

These observations are consistent with the behavior of the encoding functions on sparse columns. For Monomial, the set of columns on which $\phi_S^{M,e}$ evaluates to one can only shrink when a coordinate is added to $S$. If the columns are sparse, high-degree monomials tend to evaluate to zero on almost every column in both adjacent blocks. Such rows are unlikely to exhibit a large disagreement, so useful Monomial supports tend to be small.

Reflected applies the same construction to the complements of the shared variables. When the original columns are sparse, their complements are dense. Thus, a coordinate may be added to a Reflected support without changing the evaluation of the row on any active column: if $x_p^{e,i,k}=0$ for every active column $k\in\mathcal K^{i}$ with $\phi_{S}^{\mathrm{R},e}(\mathbf{x}^{e,i,k})=1$, then $\phi_{S\cup\{p\}}^{\mathrm{R},e}$ and $\phi_{S}^{\mathrm{R},e}$ coincide on every active column of the block. The separation problem can therefore admit several high-degree supports that are distinct as sets but induce identical or nearly identical rows on the current restricted master. The Generalized family contains the Reflected family as the special case $S^{+}=\emptyset$ and may inherit this behavior.

The comparison between Generalized and Reflected signatures supports this interpretation. Table~\ref{tab:generalized-reflected-overlap} compares the rows generated by the two methods on the same instance, coupling edge, and first outer iteration. A Generalized signature is called \emph{negative-only} when $S^{+}=\emptyset$ and $S^{-}\neq\emptyset$; its unsigned support is then $S^{-}$. The column ``Negative-only'' reports the fraction of Generalized signatures of this form, while ``Exact match'' reports the fraction of negative-only signatures whose unsigned support coincides with the support of a Reflected row generated on the same edge. The last two columns report the mean Jaccard distance from each unsigned Generalized support to its closest Reflected support, first over all Generalized signatures and then over the negative-only signatures. Entries are averaged over signatures within each edge, over edges within each instance, and finally across instances. An edge contributes to a column only when the corresponding quantity is defined.

\begin{table}[t]
\centering
\small
\caption{Overlap between Generalized and Reflected signatures.}
\label{tab:generalized-reflected-overlap}
\begin{tabular}{lrrrrr}
\toprule
Problem & Instances & Negative only & Exact match & Nearest dist. & Nearest neg. \\
\midrule
Stable set & 74 & 0.736 & 0.682 & 0.145 & 0.055 \\
Dominating set & 77 & 0.839 & 0.574 & 0.129 & 0.060 \\
\bottomrule
\end{tabular}
\end{table}

The fraction of first-iteration Generalized signatures that are negative-only is $0.736$ for stable set and $0.839$ for dominating set. Among the negative-only signatures, the mean fraction whose unsigned support exactly matches that of a Reflected row generated on the same edge is $0.682$ and $0.574$, respectively. More generally, negative-only signatures remain close to the Reflected supports: their mean nearest-neighbor Jaccard distance is $0.055$ for stable set and $0.060$ for dominating set, compared with $0.145$ and $0.129$ when all Generalized signatures are included. Thus, in the first iteration, Generalized behaves predominantly, though not exclusively, like Reflected at the level of unsigned supports; signatures containing positive literals tend to use supports farther from those generated by Reflected.

Overall, the stored signatures provide empirical support for the mechanism suggested by the heatmaps. Monomial generates fewer and sparser rows, but their signatures are more dispersed in literal space. Reflected and Generalized generate many dense rows concentrated around common patterns, and Generalized largely reproduces the negative-literal structure of Reflected. This helps explain why their high cut counts do not yield a corresponding improvement in convergence.

\section{Discussion}
\label{sec:discussion}

This paper develops a computational realization of strengthened Lagrangian dual for tree-structured integer programs. The main idea is to work with the primal convex characterization of the strengthened dual and to generate both of its potentially exponential objects dynamically: feasible local solutions are generated as columns through blockwise pricing, while consistency constraints using nonlinear encoding functions are generated as rows through separation. With a complete exact encoding family, this gives an exact algorithm; moreover before convergence, completed outer iterations provide valid dual bounds. Thus, exactness does not require explicitly constructing the exponentially large formulation apriori, and local optimization remains decomposed across the blocks throughout the algorithm.

Computationally, CRG consistently outperforms the monolithic formulation on stable-set and dominating-set benchmarks, closing more instances and producing stronger bounds, while applying without modification across star, path, binary-tree, and random-tree topologies. On the star instances, specialized integer L-shaped and scenario-decomposition methods are generally marginally faster, as expected from methods designed specifically for a master--recourse structure. Nevertheless, CRG remains competitive, particularly on the harder dominating-set instances, despite making no use of a distinguished first-stage block or recourse value function. The intended role of the framework is therefore not to replace specialized methods when additional structure is available, but to provide a common decomposition mechanism when it is not.

The experiments suggest that interface size is an important driver of difficulty. Performance deteriorates as the number of shared variables increases, especially for dominating set. This is consistent with the formulation: local complexity is handled independently inside the pricing oracles, whereas the master must progressively recover the information lost by separating the blocks across their shared boundaries.

This leads to perhaps the most important conclusion concerning the choice of encoding. The theoretical results show that Generalized encodings can require exponentially fewer rows than Vertex, Monomial, or Reflected encodings on suitable instances, while even the Generalized family can require exponentially many rows when a single problem-specific encoding suffices. The computations reveal a complementary phenomenon: greater expressive power does not necessarily translate into a better algorithm. Generalized and Reflected separation frequently produces many dense and closely related rows, whereas Monomial generates substantially fewer, sparser, and more diverse rows. Consequently, the theoretically richer Generalized family performs worse than the simpler Monomial and Vertex strategies on the present benchmarks. What matters computationally is therefore not only whether a family can eventually close the gap, but whether separation can identify a small collection of informative rows whose strengthening effect justifies the additional complexity they introduce into subsequent pricing problems.

These observations suggest several directions for future research. First, encoding selection itself should be adaptive. The Hybrid results already show that changing the encoding strategy during the computation can be beneficial. More sophisticated schemes could use the current boundary distributions, violations, dual multipliers, or the observed effectiveness and redundancy of previous rows to choose which encoding family to search and which rows to retain. Second, there is considerable potential in discovering problem-specific encodings. The parity example demonstrates that the difference between a generic family and an appropriate problem-specific encoding can be exponential. 


More broadly, the results highlight the information exchanged across block boundaries as a key determinant of decomposition performance. Nonlinear edge encodings capture this information, but different encodings can require dramatically different numbers of constraints to recover exactness.


\bibliographystyle{plain}
\bibliography{ref}

@article{bergner2015automatic,
  author  = {Bergner, Martin and Caprara, Alberto and Ceselli, Alberto and
             Furini, Fabio and L{\"u}bbecke, Marco E. and Malaguti, Enrico and
             Traversi, Emiliano},
  title   = {Automatic {D}antzig--{W}olfe reformulation of mixed integer programs},
  journal = {Mathematical Programming},
  volume  = {149},
  number  = {1--2},
  pages   = {391--424},
  year    = {2015},
  doi     = {10.1007/s10107-014-0761-5}
}

@article{FaenzaMunozPokutta2022,
  author  = {Yuri Faenza and Gonzalo Mu{\~n}oz and Sebastian Pokutta},
  title   = {New limits of treewidth-based tractability in optimization},
  journal = {Mathematical Programming},
  volume  = {191},
  number  = {2},
  pages   = {559--594},
  year    = {2022}
}

@book{tan-dm,
  author    = {Pang-Ning Tan and Michael Steinbach and Anuj Karpatne and Vipin Kumar},
  title     = {Introduction to Data Mining},
  edition   = {2nd},
  publisher = {Pearson},
  year      = {2019}}

@article{deng-xie,
  author  = {Haoyun Deng and Weijun Xie},
  title   = {On the {R}e{LU} {L}agrangian cuts for stochastic mixed integer programming},
  journal = {arXiv preprint arXiv:2411.01229},
  year    = {2024}}

@article{bdgl-ijoc,
  author  = {Merve Bodur and Sanjeeb Dash and Oktay G{\"u}nl{\"u}k and James Luedtke},
  title   = {Strengthened {B}enders cuts for stochastic integer programs with continuous recourse},
  journal = {INFORMS Journal on Computing}, volume = {29}, number = {1},
  pages   = {77--91}, year = {2017}}

@article{zas-sddip,
  author  = {Jikai Zou and Shabbir Ahmed and Xu Andy Sun},
  title   = {Stochastic dual dynamic integer programming},
  journal = {Mathematical Programming}, volume = {175}, number = {1},
  pages   = {461--502}, year = {2019}}

@article{cl-ijoc,
  author  = {Rui Chen and James Luedtke},
  title   = {On generating {L}agrangian cuts for two-stage stochastic integer programs},
  journal = {INFORMS Journal on Computing}, volume = {34}, number = {4},
  pages   = {2332--2349}, year = {2022}}

@article{acp-lipschitz,
  author  = {Shabbir Ahmed and Filipe Goulart Cabral and Bernardo Freitas Paulo da Costa},
  title   = {Stochastic {L}ipschitz dynamic programming},
  journal = {Mathematical Programming}, volume = {191}, number = {2},
  pages   = {755--793}, year = {2022}}

@article{anstreicher2009two,
  title={Two “well-known” properties of subgradient optimization},
  author={Anstreicher, Kurt M. and Wolsey, Laurence A.},
  journal={Mathematical Programming},
  volume={120},
  number={1},
  pages={213--220},
  year={2009},
  publisher={Springer}
}

@book{minsky-papert,
  author    = {Minsky, Marvin and Papert, Seymour},
  title     = {Perceptrons: An Introduction to Computational Geometry},
  publisher = {MIT Press},
  address   = {Cambridge, MA},
  year      = {1969}
}

@article{caroe-schultz,
  author  = {Car{\o}e, Claus C. and Schultz, R{\"u}diger},
  title   = {Dual decomposition in stochastic integer programming},
  journal = {Operations Research Letters},
  volume  = {24},
  number  = {1--2},
  pages   = {37--45},
  year    = {1999}
}

@article{rockafellar-wets,
  author  = {Rockafellar, R. Tyrrell and Wets, Roger J.-B.},
  title   = {Scenarios and policy aggregation in optimization under uncertainty},
  journal = {Mathematics of Operations Research},
  volume  = {16},
  number  = {1},
  pages   = {119--147},
  year    = {1991}
}

@book{birge-louveaux,
  author    = {Birge, John R. and Louveaux, Fran{\c{c}}ois},
  title     = {Introduction to Stochastic Programming},
  edition   = {2nd},
  publisher = {Springer},
  address   = {New York},
  year      = {2011}
}

@article{dey2018analysis,
  title={Analysis of sparse cutting planes for sparse {MILPs} with applications to stochastic {MILPs}},
  author={Dey, Santanu S. and Molinaro, Marco and Wang, Qianyi},
  journal={Mathematics of Operations Research},
  volume={43},
  number={1},
  pages={304--332},
  year={2018},
  publisher={INFORMS}
}

@inproceedings{cdx-ipco,
  author    = {Cifuentes, Diego and Dey, Santanu S. and Xu, Jingye},
  title     = {Lagrangian Dual for Integer Optimization with Zero Duality Gap that Admits Decomposition},
  booktitle = {Integer Programming and Combinatorial Optimization (IPCO 2025)},
  series    = {Lecture Notes in Computer Science},
  pages     = {184--198},
  year      = {2025},
  publisher = {Springer}
}

@incollection{geoffrion,
  author    = {Geoffrion, Arthur M.},
  title     = {Lagrangean Relaxation for Integer Programming},
  booktitle = {Approaches to Integer Programming},
  pages     = {82--114},
  year      = {2009},
  publisher = {Springer},
  address   = {Amsterdam, Netherlands}
}

@article{guignard-kim,
  author  = {Guignard, Monique and Kim, Siwhan},
  title   = {Lagrangean Decomposition: A Model Yielding Stronger
             {L}agrangean Bounds},
  journal = {Mathematical Programming},
  volume  = {39},
  number  = {2},
  pages   = {215--228},
  year    = {1987}
}

@article{robertson-seymour,
  author  = {Robertson, Neil and Seymour, Paul D.},
  title   = {Graph Minors. {II}. {Algorithmic Aspects of Tree-Width}},
  journal = {Journal of Algorithms},
  volume  = {7},
  number  = {3},
  pages   = {309--322},
  year    = {1986}
}

@article{dantzig-wolfe,
  author  = {Dantzig, George B. and Wolfe, Philip},
  title   = {Decomposition Principle for Linear Programs},
  journal = {Operations Research},
  volume  = {8},
  number  = {1},
  pages   = {101--111},
  year    = {1960}
}

@article{barnhart-bp,
  author  = {Barnhart, Cynthia and Johnson, Ellis L. and Nemhauser,
             George L. and Savelsbergh, Martin W. P. and Vance,
             Pamela H.},
  title   = {Branch-and-Price: Column Generation for Solving Huge
             Integer Programs},
  journal = {Operations Research},
  volume  = {46},
  number  = {3},
  pages   = {316--329},
  year    = {1998}
}

@article{luebbecke-cg,
  author  = {L\"ubbecke, Marco E. and Desrosiers, Jacques},
  title   = {Selected Topics in Column Generation},
  journal = {Operations Research},
  volume  = {53},
  number  = {6},
  pages   = {1007--1023},
  year    = {2005}
}

@article{minoux-ouzia,
  author  = {Minoux, Michel and Ouzia, Hacene},
  title   = {{DRL*}: A Hierarchy of Strong Block-Decomposable Linear Relaxations for 0--1 {MIP}s},
  journal = {Discrete Applied Mathematics},
  volume  = {158},
  number  = {18},
  pages   = {2031--2048},
  year    = {2010}
}

@article{sherali-adams,
  author  = {Sherali, Hanif D. and Adams, Warren P.},
  title   = {A Hierarchy of Relaxations Between the Continuous and Convex Hull Representations for Zero-One Programming Problems},
  journal = {SIAM Journal on Discrete Mathematics},
  volume  = {3},
  number  = {3},
  pages   = {411--430},
  year    = {1990}
}

@article{lasserre,
  author  = {Lasserre, Jean B.},
  title   = {Global Optimization with Polynomials and the Problem of Moments},
  journal = {SIAM Journal on Optimization},
  volume  = {11},
  number  = {3},
  pages   = {796--817},
  year    = {2001}
}

@article{rothvoss,
  author  = {Rothvo{\ss}, Thomas},
  title   = {The {Lasserre} Hierarchy in Approximation Algorithms},
  journal = {Lecture Notes for the MAPSP},
  pages   = {1--25},
  year    = {2013}
}

@article{benders,
  author  = {Benders, Jacobus F.},
  title   = {Partitioning Procedures for Solving Mixed-Variables
             Programming Problems},
  journal = {Numerische Mathematik},
  volume  = {4},
  number  = {1},
  pages   = {238--252},
  year    = {1962}
}

@article{laporte-louveaux,
  author  = {Laporte, Gilbert and Louveaux, Fran\c{c}ois V.},
  title   = {The Integer {L}-Shaped Method for Stochastic Integer
             Programs with Complete Recourse},
  journal = {Operations Research Letters},
  volume  = {13},
  number  = {3},
  pages   = {133--142},
  year    = {1993}
}

@article{angulo-ahmed-dey,
  author  = {Angulo, Gustavo and Ahmed, Shabbir and Dey, Santanu S.},
  title   = {Improving the Integer {L}-Shaped Method},
  journal = {INFORMS Journal on Computing},
  volume  = {28},
  number  = {3},
  pages   = {483--499},
  year    = {2016}
}

@article{ahmed-scenario,
  author  = {Ahmed, Shabbir},
  title   = {A Scenario Decomposition Algorithm for 0--1 Stochastic
             Programs},
  journal = {Operations Research Letters},
  volume  = {41},
  number  = {6},
  pages   = {565--569},
  year    = {2013}
}

@article{boland-ph,
  author  = {Boland, Natashia and Christiansen, Jeffrey and Dandurand,
             Brian and Eberhard, Andrew and Linderoth, Jeff and Luedtke,
             James and Oliveira, Fabricio},
  title   = {Combining Progressive Hedging with a {Frank--Wolfe} Method
             to Compute {Lagrangian} Dual Bounds in Stochastic
             Mixed-Integer Programming},
  journal = {SIAM Journal on Optimization},
  volume  = {28},
  number  = {2},
  pages   = {1312--1336},
  year    = {2018}
}

@article{feizollahi-al,
  author  = {Feizollahi, Mohammad Javad and Ahmed, Shabbir and Sun,
             Andy},
  title   = {Exact Augmented {Lagrangian} Duality for Mixed Integer
             Linear Programming},
  journal = {Mathematical Programming},
  volume  = {161},
  number  = {1--2},
  pages   = {365--387},
  year    = {2017}
}

@article{gu-al,
  author  = {Gu, Xiaoyi and Ahmed, Shabbir and Dey, Santanu S.},
  title   = {Exact Augmented {Lagrangian} Duality for Mixed Integer
             Quadratic Programming},
  journal = {SIAM Journal on Optimization},
  volume  = {30},
  number  = {1},
  pages   = {781--797},
  year    = {2020}
}

@book{beltran-sl,
  author    = {Beltran, C\'esar and Tadonki, Claude and Vial,
               Jean-Philippe},
  title     = {Semi-Lagrangian Relaxation},
  publisher = {HEC Gen\`eve},
  address   = {Geneva, Switzerland},
  year      = {2004}
}

@article{sun-decomp,
  author  = {Sun, Kaizhao and Sun, Mou and Yin, Wotao},
  title   = {Decomposition Methods for Global Solution of Mixed-Integer
             Linear Programs},
  journal = {SIAM Journal on Optimization},
  volume  = {34},
  number  = {2},
  pages   = {1206--1235},
  year    = {2024}
}

@article{bertele-brioschi,
  author  = {Bertel\`e, Umberto and Brioschi, Francesco},
  title   = {On Non-Serial Dynamic Programming},
  journal = {Journal of Combinatorial Theory, Series A},
  volume  = {14},
  number  = {2},
  pages   = {137--148},
  year    = {1973}
}

@article{bienstock-munoz,
  author  = {Bienstock, Daniel and Mu{\~n}oz, Gonzalo},
  title   = {{LP} Formulations for Polynomial Optimization Problems},
  journal = {SIAM Journal on Optimization},
  volume  = {28},
  number  = {2},
  pages   = {1121--1150},
  year    = {2018}
}

\section*{Statements and declarations}


\subsection*{Competing Interests}
The authors have no relevant financial or non-financial interests to disclose.

\subsection*{Data Availability}
Code and results are available at \texttt{https://github.com/ganguloo/block-lagrangian}.

\end{document}